\documentclass[11pt, twoside, leqno]{article}

\usepackage{amssymb}
\usepackage{amsmath}
\usepackage{amsthm}
\usepackage{color}
\usepackage{mathrsfs}

\allowdisplaybreaks	

\def\rr{{\mathbb R}}
\def\rn{{{\rr}^n}}
\def\zz{{\mathbb Z}}
\def\cc{{\mathbb C}}
\def\nn{{\mathbb N}}

\def\ca{{\mathcal A}}

\def\cs{{\mathcal S}}

\def\fz{\infty}
\def\az{\alpha}
\def\bz{\beta}
\def\dz{\delta}
\def\bdz{\Delta}

\def\gz{{\gamma}}
\def\bgz{{\Gamma}}

\def\lz{\lambda}
\def\blz{\Lambda}

\def\tz{\theta}

\def\vz{\varphi}

\def\lf{\left}
\def\r{\right}

\def\la{\langle}
\def\ra{\rangle}
\def\hs{\hspace{0.25cm}}
\def\ls{\lesssim}

\def\pa{\partial}

\def\noz{\nonumber}
\def\wz{\widetilde}
\def\wh{\widehat}
\def\st{\subset}
\def\com{\complement}
\def\bh{\backslash}

\def\dist{\mathop\mathrm{\,dist\,}}
\def\supp{\mathop\mathrm{\,supp\,}}

\def\loc{\mathop\mathrm{\,loc\,}}
\def\div{\mathop\mathrm{div}}
\def\essinf{\mathop\mathrm{\,ess\,inf\,}}

\def\dint{\displaystyle\int}

\def\dsup{\displaystyle\sup}

\def\rnn{\rr_+^{n+1}}
\def\dydt{\,\frac{dy\,dt}{t^{n+1}}}
\def\asize{(w,\,p,\,\infty)}
\def\pt{\partial}
\def\rize{H_{L_w,\,{\rm Riesz}}^p(\rn)}

\def\vez{\varepsilon}
\def\nab{\nabla}
\def\ujb{{U_j(B)}}

\newtheorem{thm}{Theorem}[section]
\newtheorem{prop}[thm]{Proposition}
\newtheorem{lem}[thm]{Lemma}
\newtheorem{cor}[thm]{Corollary}

\theoremstyle{definition}
\newtheorem{defn}[thm]{Definition}
\newtheorem{rem}[thm]{Remark}

\renewcommand{\vec}[1]{\boldsymbol{#1}}
\numberwithin{equation}{section}

\begin{document}

\arraycolsep=1pt

\title{\bf\Large Riesz Transform Characterizations of Hardy Spaces
Associated to Degenerate Elliptic Operators
\footnotetext{\hspace{-0.35cm} 2010 {\it
Mathematics Subject Classification}. Primary 47B06; Secondary 42B30, 42B35, 35J70.
\endgraf {\it Key words and phrases}. degenerate elliptic operator,
Hardy space, Hardy-Sobolev space, Riesz transform.
\endgraf Dachun Yang is supported by the National
Natural Science Foundation of China (Grant Nos. 11171027 and 11361020).  This project is also partially supported
by the Specialized Research Fund for the Doctoral Program of Higher Education
of China (Grant No. 20120003110003)  and the Fundamental Research Funds for Central
Universities of China (Grant Nos. 2013YB60 and 2014KJJCA10).}}
\author{Dachun Yang and Junqiang Zhang\,\footnote{Corresponding author}}
\date{ }
\maketitle

\vspace{-0.8cm}

%%%%-------------------------------------------------------------------
\begin{center}
\begin{minipage}{13.8cm}
{\small {\bf Abstract}\quad
Let $w$ be a Muckenhoupt $A_2(\mathbb{R}^n)$ weight
and $L_w:=-w^{-1}\mathop\mathrm{div}(A\nabla)$
the degenerate elliptic operator on the Euclidean space $\mathbb{R}^n$.
In this article, the authors establish the Riesz transform characterization
of the Hardy space $H_{L_w}^p(\mathbb{R}^n)$ associated with $L_w$,
for $w\in A_{q}(\rn)$ and $w^{-1}\in A_{2-\frac{2}{n}}(\rn)$ with $n\geq 3$,
$q\in[1,2]$
and $p\in(q(\frac{1}{r}+\frac{q-1}{2}+\frac{1}{n})^{-1},1]$
if, for some $r\in[1,\,2)$, $\{tL_w e^{-tL_w}\}_{t\geq 0}$ satisfies the weighted $L^r-L^2$
full off-diagonal estimate.}
\end{minipage}
\end{center}
%%%%-------------------------------------------------------------------

\section{Introduction}\label{s1}

\hskip\parindent
The theory of classical real Hardy spaces $H^p(\rn)$ originates
from Stein and Weiss \cite{SW60} in the early 1960s.
Since then, this real-variable theory received continuous development and
now is increasingly mature; see, for example, \cite{FS72,S93}.
It is well known that the Hardy space $H^p(\rn)$ is a
suitable substitute of the Lebesgue space $L^p(\rn)$, when
$p\in(0,\,1]$, and plays important roles in various fields of analysis and
partial differential equations.
Notice that $H^p(\rn)$ is essentially associated with the Laplace
operator $\bdz:=\sum_{j=1}^n\frac{\partial^2}{\partial x_j^2}$; see \cite{DY051,HMM11} for instance.

The motivation to study the Hardy spaces associated with different
operators (for example, the second order elliptic operator $-\div (A\nabla)$
and the Schr\"odinger operator
$-\Delta +V$) comes from characterizing the boundedness of
the associated Riesz transforms and the regularity of solutions
of the associated equations; see, for example, \cite{ADM05,DY05,DY051,AMR08,Yan08,JY10,HM09,
HMM11,dl13,dhmmy13,ccyy13}.

Consider now a degenerate elliptic operator.
Let $w$ be a Muckenhoupt $A_2(\mathbb{R}^n)$ weight and
$A(x):=(A_{ij}(x))_{i,j=1}^n$ be a matrix of complex-valued, measurable
functions on $\rn$ satisfying the \emph{degenerate elliptic condition}
that there exist positive constants $\lz\le\blz$ such that,
for almost every $x\in\rn$ and all $\xi$, $\eta\in\cc^n$,
\begin{eqnarray}\label{degenerate C1}
\lf|\langle A(x) \xi,\,\eta \rangle\r| \le \blz w(x)|\xi||\eta|
\end{eqnarray}
and
\begin{eqnarray}\label{degenerate C2}
\Re \langle A(x) \xi,\,\xi \rangle \geq \lz w(x)|\xi|^2,
\end{eqnarray}
here and hereafter, $\Re z$ denotes the \emph{real part} of $z$ for any $z\in\cc$.
The associated \emph{degenerate elliptic operator} $L_w$ is defined by setting,
for all $f\in D(L_w)\st\mathcal{H}_0^1(w,\,\rn)$,
\begin{eqnarray}\label{Lw}
L_wf:=-\frac{1}{w}\div (A\nabla f),
\end{eqnarray}
which is interpreted in the usual weak sense via the sesquilinear form,
where $D(L_w)$ denotes the domain of $L_w$.
Here and in what follows, $\mathcal{H}_0^1(w,\,\rn)$ denotes the \emph{weighted Sobolev space},
which is defined to be the closure of $C_c^\fz(\rn)$ with respect to the \emph{norm}
$$\|f\|_{\mathcal{H}_0^1(w,\,\rn)}:=\lf\{\int_\rn \lf[|f(x)|^2+|\nabla f(x)|^2\r]w(x)\,dx\r\}^{1/2}.$$
The \emph{sesquilinear form} $\mathfrak{a}$ associated with $L_w$ is defined by setting,
for all $f,\,g\in\mathcal{H}_0^1(w,\,\rn)$,
\begin{eqnarray*}%\label{sesqui form}
\mathfrak{a}(f,\,g):=\dint_{\rn}[A(x)\nabla f(x)]\cdot \overline{\nabla g(x)}\,dx.
\end{eqnarray*}

In the case $w\equiv 1$, the degenerate elliptic operator $L_w$ reduces to
the usual second order elliptic operator $L=-\div (A\nabla)$.
Therefore, $L_w$ may be considered as a generalization of the usual uniformly
elliptic operator.

Operators of the form \eqref{Lw} and the associated elliptic equations
were first considered by Fabes, Kenig and Serapioni \cite{FKS82}
and have been considered by a number of other authors (see, for example,
\cite{CF84,CS85,CF87,Ka99} and, especially, some recent articles by Cruz-Uribe and Rios \cite{CR08,CR12,CR13,CR14}).
We point out that, when $w$ is a weight in the Muchkenhoupt class $A_2(\rn)$, the space
$\mathcal{H}_0^1(w,\,\rn)$ was first studied by Fabes et al. in \cite{FKS82}, where the local weighted
Sobolev embedding theorem and the Poincar\'e inequality
were proved to hold true.

Let $L_w$ be a degenerate elliptic operator
as in \eqref{Lw} with $w$ in the Muckenhoupt class
of $A_2(\rn)$ weights (see Subsection \ref{s2.0}
below for their exact definitions).
The main purpose of this article is to establish the Riesz
transform characterizations of Hardy spaces $H_{L_w}^p(\rn)$ associated with $L_w$
(see Theorem \ref{thm main} below).

This article may be viewed in part as a sequel to \cite{ZCJY14}, where
the non-tangential maximal function characterizations of Hardy spaces $H_{L_w}^p(\rn)$
associated with $L_w$ and the boundedness of the associated Riesz transform on these spaces
have been studied.

To state the main results of this article,
we first introduce some definitions and notation.
Let $w\in A_2(\rn)$, $L_w$ be as in \eqref{Lw} and
$f\in L^2(w,\,\rn)$, where $L^2(w,\,\rn)$ denotes the \emph{weighted Lebesgue space}
with the \emph{norm}
\begin{eqnarray*}
\|f\|_{L^2(w,\,\rn)}:=\lf\{\dint_\rn\lf|f(x)\r|^2w(x)\,dx\r\}^{\frac{1}{2}}.
\end{eqnarray*}
It is well known that, if $w\in A_2(\rn)$, $L^2(w,\,\rn)$ is a space of homogenous type
in the sense of Coifman and Weiss \cite{CW71,CW77}, since $w(x)\,dx$ is a doubling measure.
In what follows, let $\rr_+^{n+1}:=\rr^n\times(0,\fz)$.
For any $f\in L^2(w,\,\rn)$ and $x\in\rn$, the \emph{square function $\cs_{L_w}(f)$
associated with} $L_w$ is defined by setting
\begin{eqnarray*}%\label{square function}
\cs_{L_w}(f)(x):=\lf[\iint_{\bgz(x)}\lf|t^2L_we^{-t^2L_w}(f)(y)\r|^2w(y)\,
\frac{dy}{w(B(x,t))}\,\frac{dt}{t}\r]^{1/2},
\end{eqnarray*}
where $B(x,t):=\{y\in\rn:\ |x-y|<t\}$,
$w(B(x,t)):=\int_{B(x,t)}w(y)\,dy$
and
\begin{eqnarray}\label{cone}
\bgz_\az(x):=\{(y,t)\in\rr^{n+1}_+:\ |x-y|<\az t\}
\end{eqnarray}
denotes the \emph{cone of aperture} $\az$ \emph{with
vertex} $x$. In particular, if $\az=1$, we write $\Gamma(x)$
instead of $\Gamma_\az(x)$.

The Hardy spaces $H_{L_w}^p(\rn)$ associated with $L_w$ were defined in
\cite[Definition 1.1]{ZCJY14} as follows.
\begin{defn}[\cite{ZCJY14}]\label{def Hardy space}
Let $p\in(0,\,1]$, $w\in A_2(\rn)$
and $L_w$ be the degenerate elliptic operator
as in \eqref{Lw} with the matrix $A$ satisfying the degenerate elliptic
conditions \eqref{degenerate C1} and \eqref{degenerate C2}.
The \emph{Hardy space} $H_{L_w}^p(\rn)$, \emph{associated with}
$L_w$, is defined as the completion of the space
\begin{eqnarray*}
\lf\{f\in L^2(w,\,\rn):\ \|S_{L_w}(f)\|_{L^p(w,\,\rn)}<\fz\r\}
\end{eqnarray*}
with respect to the \emph{quasi-norm}
\begin{eqnarray*}
\lf\|f\r\|_{H_{L_w}^p(\rn)}:=\|S_{L_w}(f)\|_{L^p(w,\,\rn)}.
\end{eqnarray*}
\end{defn}

We introduce the following Hardy spaces associated with the Riesz transform,
which, when $w\equiv1$, is just the one defined in \cite[p.\,7]{HMM11}.
\begin{defn}\label{def Riesz Hardy}
Let $p\in(0,\,1]$, $w\in A_2(\rn)$
and $L_w$ be the degenerate elliptic operator
as in \eqref{Lw} with the matrix $A$ satisfying the degenerate elliptic
conditions \eqref{degenerate C1} and \eqref{degenerate C2}.
The \emph{Hardy space} $H_{L_w,\,{\rm Riesz}}^p(\rn)$ is defined as the completion of the space
\begin{eqnarray*}
\lf\{f\in L^2(w,\,\rn):\ \nabla L_w^{-1/2}f\in H_w^p(\rn)\r\}
\end{eqnarray*}
with respect to the \emph{quasi-norm}
\begin{eqnarray*}
\lf\|f\r\|_{H_{L_w,\,{\rm Riesz}}^p(\rn)}:=\|\nabla L_w^{-1/2}f\|_{H^p_w(\rn)}.
\end{eqnarray*}
\end{defn}

Before establishing the Riesz transform characterization of $H_{L_w}^p(\rn)$,
we first introduce the following definition of weighted full off-diagonal estimates,
which is a generalization of full off-diagonal estimates in spaces of homogeneous
type (see \cite[Definition 3.1]{am07ii}).
\begin{defn}\label{def off-diagonal}
Let $w\in A_\fz(\rn)$ and $1\le p\le q<\fz$. A family $\{T_t\}_{t\geq 0}$
of sublinear operators is said to satisfy the \emph{weighted $L^p-L^q$ full off-diagonal estimates},
denoted by $T_t\in \mathcal{F}_w(L^p-L^q)$,
if there exist positive constants $C,\,c\in (0,\fz)$ such that, for any closed sets
$E,\,F$ of $\rn$ and $f\in L^p(w^{\frac{p}{2}},\,E)$ with $\supp f\st E$,
\begin{eqnarray*}
\lf\{\int_F|T_t(f)(x)|^q [w(x)]^{\frac q2}\,dx\r\}^{\frac 1q}
\le C t^{-\frac{n}{2}\lf(\frac 1p-\frac 1q\r)}e^{-c\frac{[d(E,\,F)]^2}{t}}
\lf\{\int_E |f(x)|^p[w(x)]^{\frac p2}\,dx\r\}^{\frac 1p}.
\end{eqnarray*}
\end{defn}

The following theorem establishes the Riesz transform characterizations of $H_{L_w}^p(\rn)$.
\begin{thm}\label{thm main}
Let $q\in[1,2]$ and $w\in A_{q}(\rn)$ satisfy $w^{-1}\in A_{2-\frac{2}{n}}(\rn)$ with $n\geq 3$.
Assume $tL_w e^{-tL_w}\in\mathcal{F}_w(L^r-L^2)$ for some $r\in[1,\,2)$.
Then, for $p\in(q(\frac{1}{r}+\frac{q-1}{2}+\frac{1}{n})^{-1},1]$,
the Hardy spaces $H_{L_w,\,{\rm Riesz}}^p(\rn)$ and $H_{L_w}^p(\rn)$ coincide with
equivalent quasi-norms.
\end{thm}

\begin{rem}\label{rem main}
(i) Since we need to apply the weighted Sobolev inequality (see \eqref{eq Sobolev} below)
in the proof of Theorem \ref{thm main}, to this end, we need to assume
$w^{-1}\in A_{2-\frac{2}{n}}(\rn)$ with $n\geq 3$ in Theorem \ref{thm main}.

(ii) In the case $w\equiv1$, Theorem \ref{thm main} reduces to \cite[Proposition 5.18]{HMM11},
where Hofmann et al. first established the Riesz transform characterizations of
Hardy spaces $H_L^p(\rn)$ associated with the second order elliptic operators
$L:=-\div(A\nabla)$; we point out that, in this case, the range of $p$ of Theorem
\ref{thm main} coincides with that of \cite[Proposition 5.18]{HMM11}.
\end{rem}

From  Theorem \ref{thm main} and Remark \ref{rem off} below,
we immediately deduce the following conclusions, the details being omitted.

\begin{cor}\label{cor1}
Let $q\in[1,2]$, $s\in (1,\fz]$ and $w\in A_{q}(\rn)\cap RH_s(\rn)$
satisfy $w^{-1}\in A_{2-\frac{2}{n}}(\rn)$ with $n\geq 3$.
If the matrix $A$ associated with $L_w$ is real symmetric, then, for all
$p\in(q[\frac 12(1-\frac1s)+\frac q2+\frac{1}{n}]^{-1},1]$,
the Hardy spaces $H_{L_w,\,{\rm Riesz}}^p(\rn)$ and $H_{L_w}^p(\rn)$ coincide with
equivalent quasi-norms.
\end{cor}

By Theorem \ref{thm main} and Proposition \ref{pro off-diagonal} below,
we immediately conclude the following conclusion, the details being omitted.

\begin{cor}\label{cor2}
Let $q\in[1,\frac43)$ and $w\in A_{q}(\rr^3)$ satisfy $w^{-1}\in A_{\frac43}(\rr^3)$.
Then, for $p\in(\frac{6q}{4+3q},1]$,
the Hardy spaces $H_{L_w,\,{\rm Riesz}}^p(\rr^3)$ and $H_{L_w}^p(\rr^3)$ coincide with
equivalent quasi-norms.
\end{cor}

We prove Theorem \ref{thm main} by following the strategy used in \cite[Proposition 5.18]{HMM11}.
The proof of Theorem \ref{thm main} rests on the atomic decomposition of the weighted
Hardy-Sobolev spaces (see Theorem \ref{pro second} below).

In what follows, let $\cs(\rn)$ denote the \emph{space of all Schwartz functions}
and $\cs^\prime(\rn)$ be the \emph{space of all Schwartz distributions}.

Let $\psi\in\cs(\rn)$, $\int_\rn \psi(x)\,dx=1$ and $\psi_t(x):=t^{-n}\psi(\frac{x}{t})$
for all $x\in\rn$ and $t\in(0,\fz)$.
For all $f\in\cs^\prime(\rn)$ and $x\in\rn$,
the \emph{non-tangential maximal function}
$\psi^\ast_\nabla(f)(x)$ is defined by setting
$$\psi^\ast_\nabla(f)(x):=\sup_{|x-y|<t,\ t\in (0,\fz)}|(\psi_t\ast f)(y)|.$$
Then, for $p\in(0,1]$ and $w\in A_\fz(\rn)$, $f\in\cs^\prime(\rn)$ is said
to belong to the \emph{weighted Hardy space} $H_w^p(\rn)$,
if $\psi^\ast_\nabla(f)\in L^p(w,\,\rn)$; moreover, its norm is given by
$\|f\|_{H_w^p(\rn)}:=\|\psi^\ast_\nabla(f)\|_{L^p(w,\,\rn)}$.

Let $\mathcal{S}_0(\rn)$ be the space of all Schwartz functions $\vz$ that
satisfy $\int_\rn \vz(x)\,dx=0$.
Then $\mathcal{S}_0(\rn)$ is a subspace of $\mathcal{S}(\rn)$ that
inherits the same topology as $\mathcal{S}(\rn)$. We denote
the dual of $\mathcal{S}_0(\rn)$ by $\mathcal{S}'_0(\rn)$.

Let $p\in (0,1]$ and $w\in A_\fz(\rn)$.
The \emph{weighted Hardy-Sobolev space} is defined as the set
\begin{equation*}
H^{1,p}_w(\rn):=\{f\in \mathcal{S}'_0(\rn):\ \nab f\in H^p_w(\rn)\}
\end{equation*}
with the quasi-norm
\begin{equation*}
\|f\|_{H^{1,p}_w(\rn)}:=\|\nab f\|_{H^p_w(\rn)}:=\sum_{j=1}^n\|\partial_j f\|_{H^p_w(\rn)},
\end{equation*}
where $\nab f:=(\partial_1f,\,\ldots,\,\partial_nf)$ stands for the
distributional derivatives of $f$ and $\nab f\in H^p_w(\rn)$ means
that, for all $j\in \{1,\,\ldots,\,n\}$, $\partial_jf\in H^p_w(\rn)$.

In what follows, for a subset $E\st\rn$, $C^\fz_c(E)$ denotes the set of all $C^\fz$ functions
with compact support in $E$.
For a ball $B$ of $\rn$ and $w\in A_\fz(\rn)$, we define $H_0^1(w,\,B)$ to be
the closure of $C^\fz_c(B)$ with respect to the norm
$$\|f\|_{H_0^1(w,\,B)}:=\lf\{\int_B \lf[|f(x)|^2+|\nab f(x)|^2\r]w(x)\,dx\r\}^{\frac12}.$$

Let  $p\in (0,1]$, $w\in A_\fz(\rn)$ and $B\st\rn$ be a ball.
A function $a\in H_0^1(w,\,B)$ is called an \emph{$H^{1,p}_w(\rn)$-atom} if

(i) $\supp a\st B$;

(ii) $\|a\|_{L^2(w,\,B)}\le r_B \|\nab a\|_{L^2(w,\,B)}$, where $r_B$ denotes the radius of $B$;

(iii) $\|\nab a\|_{L^2(w,\,B)}\le [w(B)]^{\frac12-\frac1p}$.

The following theorem gives an atomic decomposition for distributions in $H^{1,p}_w(\rn)$,
which plays a key role in the proof of Theorem \ref{thm main}.
\begin{thm}\label{pro second}
Let $w\in A_2(\rn)$, $p\in (0,1]$ and $f\in \mathcal{H}_0^1(w,\,\rn)\cap H^{1,p}_w(\rn)$.
Then there exist a sequence of $H^{1,p}_w(\rn)$-atoms, $\{\bz_k\}_{k\in\nn}$, and
a sequence of numbers, $\{\lz_k\}_{k\in\nn}\st\cc$, such that
\begin{equation}\label{eq 2.xx}
f=\sum_{k=1}^\fz \lz_k \bz_k\ \ \ \ \text{in}\ \mathcal{S}'_0(\rn),
\end{equation}
and
\begin{equation*}
\nab f=\sum_{k=1}^\fz \lz_k \nab \bz_k\ \ \ \ \text{in}\ L^2(w,\,\rn).
\end{equation*}
Moreover, there exists a positive constant $C$, independent of $f$, such that
\begin{equation*}
\lf\{\sum_{k=1}^\fz |\lz_k|^p\r\}^{\frac{1}{p}}\le C\|\nab f\|_{H_w^p(\rn)}.
\end{equation*}
\end{thm}

Recall that Lou and Yang in \cite{LY07} gave an atomic characterization for the
classical Hardy-Sobolev space $H^{1,1}(\rn)$. Following their methods therein,
we prove Theorem \ref{pro second} through the atomic decomposition for tent spaces,
which was originally introduced in \cite{CMS85}.
However, we point out that the proof of Theorem \ref{pro second}
is slightly different from that of \cite[Lemma 1]{LY07}. We prove the size condition of
the $H^{1,p}_w(\rn)$-atoms by the local weighted Sobolev embedding theorems in
\cite{FKS82}, for $A_2(\rn)$ weights, instead of \cite[Chapter 3,\,Theorem 3.3.3]{S95}
which was used in the corresponding proof of \cite[Lemma 1]{LY07}.

This article is organized as follows. In Subsection \ref{s2.0}, we first recall
some notions and results on Muckenhoupt weights; then, in Subsection \ref{s2.1},
we establish the weighted off-diagonal estimates for $L_w$; in Subsection \ref{s2.2},
we introduce the weighted tent space and establish its atomic decomposition.
Section \ref{s3} is devoted to the proof of Theorem \ref{pro second},
while Theorem \ref{thm main} is proved in Section \ref{s4}.

We end this section by making some conventions on notation. Throughout this article,
$L_w$ always denotes a degenerate elliptic operator as in \eqref{Lw}.
We denote by $C$ a positive constant which is independent of the main parameters,
but it may vary from line to line.
We also use $C_{(\az, \bz,\ldots)}$ to denote a positive constant depending on
the parameters $\az$, $\bz$, $\ldots$.
The \emph{symbol $f\ls g$} means that $f\le Cg$.
If $f\ls g$ and $g\ls f$, then we write $f\sim g$.
For any measurable subset $E$ of $\rn$, we denote by $E^\com$ the \emph{set $\rn\bh E$}.
Let $\nn:=\{1,\,2,\,\ldots\}$ and $\zz_+:=\nn\cup\{0\}$.
For any closed set $F\st\rn$, we let
\begin{equation}\label{eq tent}
R(F):=\bigcup_{x\in F}\bgz(x),
\end{equation}
where $\bgz(x)$ for all $x\in F$ is as in \eqref{cone} with $\az=1$.
For any ball $B:=(x_B,r_B)\st\rn$, $\az\in(0,\fz)$ and $j\in\nn$,
we let $\az B:=B(x_B,\az r_B)$,
\begin{eqnarray}\label{eq-def of ujb}
U_0(B):=B\ \ \ \text{and}\ \ \ U_j(B):=(2^jB)\setminus (2^{j-1}B).
\end{eqnarray}

\section{Preliminaries}\label{s2}
\hskip\parindent
In this section, we first recall the definition of the \emph{Muckenhoupt weights} and
some of their properties. Then we establish the weighted full off-diagonal estimates for
$L_w$, which play a key role in the proofs of our main results.
Finally, we recall the definition of the weighted tent space and its atomic decomposition, which
is used in Section \ref{s3}.

\subsection{Muckenhoupt weights}\label{s2.0}

\hskip\parindent
Let $q\in[1,\fz)$. A nonnegative and locally integrable function $w$ on $\rn$
is said to belong to the \emph{Muckenhoupt class} $A_q(\rn)$,
if there exists a positive constant $C$ such that, for all balls $B\st\rn$, when $q\in(1,\fz)$,
$$\frac{1}{|B|}\int_B w(x)\,dx\lf\{\frac{1}{|B|}\int_B [w(x)]^{-\frac{1}{q-1}}\,dx\r\}^{q-1}\le C$$
or, when $q=1$,
$$\frac{1}{|B|}\int_B w(x)\,dx\le C\essinf_{x\in B}w(x).$$
We also let $A_\fz(\rn):=\cup_{q\in[1,\fz)}A_q(\rn)$ and $w(E):=\int_E w(x)\,dx$
for any measurable set $E\st\rn$.

Let $r\in(1,\fz]$. A nonnegative locally integrable function $w$ is said to
belong to the \emph{reverse H\"older class $RH_r(\rn)$}, if there exists
a positive constant $C$ such that, for all balls $B\st\rn$,
\begin{eqnarray*}
\lf\{\frac{1}{|B|}\int_B [w(x)]^r\,dx\r\}^{1/r}\le C\frac{1}{|B|}\int_B w(x)\,dx,
\end{eqnarray*}
where we replace  $\{\frac{1}{|B|}\int_B [w(x)]^r\,dx\}^{1/r}$ by $\|w\|_{L^\infty(B)}$ when $r=\infty$.

We recall some properties of Muckenhoupt weights and reverse H\"older classes
in the following two lemmas (see, for example, \cite{Du01} for their proofs).
\begin{lem}\label{lem Ap-1}
{\rm(i)} If $1\le p\le q\le\fz$, then $A_1(\rn)\st A_p(\rn)\st A_q(\rn)$.

{\rm(ii)} $A_\fz(\rn):=\cup_{p\in[1,\fz)}A_p(\rn)=\cup_{r\in(1,\fz]}RH_r(\rn)$.
\end{lem}

\begin{lem}\label{lem Ap-2}
Let $q\in[1,\fz)$ and $r\in (1,\fz]$. If a nonnegative measurable function
$w\in A_q(\rn)\cap RH_r(\rn)$, then there exists a constant $C\in(1,\fz)$
such that, for all balls $B\st\rn$ and any measurable subset $E$ of $B$,
$$C^{-1}\lf(\frac{|E|}{|B|}\r)^q\le\frac{w(E)}{w(B)}\le C\lf(\frac{|E|}{|B|}\r)^{\frac{r-1}{r}}.$$
\end{lem}

\subsection{Weighted full off-diagonal estimates for $L_w$}\label{s2.1}
\hskip\parindent
In this subsection, we first recall the definition of weighted off-diagonal estimates on balls
from \cite{am07ii}. Then we show that, if the matrix $A$ associated with $L_w$
is real symmetric,
then, for any $p\in[1,\,2)$, $tL_we^{-tL_w}\in\mathcal{F}_w(L^p-L^2)$.
Finally, we prove that, in the general case,
for $n\geq 3$, $k\in \zz_+$ and $p_-=\frac{2n}{n+2}$,
$(tL_w)^k e^{-tL_w}\in\mathcal{F}_w(L^{p_-}-L^2)$.

\begin{defn}[\cite{am07ii}]\label{def ODEB}
Let $p,\,q\in[1,\fz]$ with $p\le q$, $w\in A_\fz(\rn)$ and $\{T_t\}_{t>0}$ be a family of sublinear
operators. The family $\{T_t\}_{t>0}$ is said to satisfy \emph{weighted $L^p$-$L^q$
off-diagonal estimates on balls}, denoted by $T_t\in \mathcal{O}_{w}(L^p-L^q)$, if
there exist constants $\theta_1,\,\theta_2 \in[0,\fz)$ and $C,\,c\in(0,\fz)$ such that,
for all $t\in(0,\fz)$ and all balls $B:=B(x_B,r_B)\subset\rn$ with $x_B\in\rn$ and $r_B\in (0,\fz)$,
and $f\in L^p_{\loc}(w,\,\rn)$,
\begin{eqnarray}\label{ODEB1}
&&\lf\{\frac{1}{w(B)}\dint_B|T_t\lf(\chi_Bf\r)(x)|^qw(x)\,dx\r\}^{1/q}\\
&&\hs\le C
\lf[\Upsilon\lf(\frac{r_B}{t^{1/2}}\r)\r]^{\theta_2}\lf\{\frac{1}{w(B)}\dint_B
|f(x)|^pw(x)\,dx\r\}^{1/p}\noz
\end{eqnarray}
and, for all $j\in\nn$ with $j\ge3$,
\begin{eqnarray*}
&&\lf\{\frac{1}{w(2^jB)}\dint_{U_j(B)}|T_t\lf(\chi_{B}f\r)(x)|^qw(x)\,dx\r\}^{1/q}\\ \nonumber
&&\hs\le C
2^{j\theta_1}\lf[\Upsilon\lf(\frac{2^jr_B}{t^{1/2}}\r)\r]^{\theta_2}
e^{-c\frac{(2^jr_B)^{2}}{t}}\lf\{\frac{1}{w(B)}\dint_B
|f(x)|^pw(x)\,dx\r\}^{1/p}
\end{eqnarray*}
and
\begin{eqnarray}\label{ODEB3}
&&\lf\{\frac{1}{w(B)}\dint_B |T_t(\chi_{U_j(B)}f)(x)|^qw(x)\,dx\r\}^{1/q}\\ \nonumber
&&\hs\le C 2^{j\theta_1}\lf[\Upsilon\lf(\frac{2^jr_B}{t^{1/2}}\r)\r]^{\theta_2}
e^{-c\frac{(2^jr_B)^{2}}{t}}\lf\{\frac{1}{w(2^jB)}\dint_{U_j(B)}
|f(x)|^pw(x)\,dx\r\}^{1/p},
\end{eqnarray}
where $U_j(B)$ is as in \eqref{eq-def of ujb} and, for all $s\in(0,\fz)$,
\begin{eqnarray*}
\Upsilon(s):=\max\lf\{s,\frac{1}{s}\r\}.
\end{eqnarray*}
\end{defn}

By borrowing some ideas from the proof of \cite[Proposition 3.2]{am07ii}, we
obtain the following conclusions.

\begin{prop}\label{pro equivalent}
Let $w\in A_\fz(\rn)\cap RH_s(\rn)$ with $s\in (1,\fz]$ and $\{T_t\}_{t\geq 0}$
be a family of sublinear operators.

(i) If $s=(1-\frac{p_0}{2})\frac{p}{p-p_0}$, $1\le p_0<p<2$ and
$T_t\in \mathcal{O}_{w}(L^{p_0}-L^2)$, then
$T_t\in\mathcal{F}_w(L^p-L^2)$.

(ii) If $s=\fz$ and $T_t\in \mathcal{O}_{w}(L^{p_0}-L^2)$ with $1\le p_0<2$,
then $T_t\in\mathcal{F}_w(L^{p_0}-L^2)$.
\end{prop}
\begin{proof}
To show (i),
let $E,\,F$ be two closed sets of $\rn$, $t\in [0,\fz)$ and
$f\in L^p(w^{\frac{p}{2}},\,E)$ with $\supp f\st E$.
We now consider two cases.

\emph{Case 1) $d(E,\,F)>0$ and $0\le t<[\frac{d(E,\,F)}{16}]^2$}.
In this case, let $r:=\frac{d(E,\,F)}{16}$ and choose a family of balls,
$B_k:=B(x_k,\,r)$ with $k\in\nn$ and $x_k\in\rn$, such that,
for any $k_1\neq k_2$, $|x_{k_1}-x_{k_2}|\geq \frac{r}{2}$ and $\cup_{k\in\nn}B_k=\rn$.
Observe that, if $x\in F$ and $y\in E$, then $|x-y|\geq d(E,\,F)=16r$.
Thus, if $x\in B_k$ for some $k\in\nn$,
then $y\notin 4B_k$, which further implies that there exists some
$j\geq 3$ such that $y\in U_j(B_k)$.
Let $\mathcal{A}:=\{k\in\nn:\ F\cap B_k\neq\emptyset\}$.
By the fact that $\supp f\st E$, the Minkowski inequality, the H\"{o}lder inequality and
\eqref{ODEB3} with $q=2$, $p=p_0$ and $B=B_k$, we see that
\begin{eqnarray}\label{eq 2.o2}
\|T_t(f)\|^2_{L^2(w,\,F)}
&&\le\sum_{k\in\ca}\int_{B_k}\lf|T_t(f)(x)\r|^2w(x)\,dx\\
&&\le\sum_{k\in\ca}\lf\{\sum_{j=3}^\fz\lf[\int_{B_k}
\lf|T_t(\chi_{U_j(B_k)}f)(x)\r|^{2}w(x)\,dx\r]^{\frac{1}{2}}\r\}^{2}\noz\\
&&\ls\sum_{k\in\ca}\lf\{\sum_{j=3}^\fz 2^{j\tz_1}\lf[\Upsilon\lf(\frac{2^jr}{\sqrt{t}}\r)\r]^{\tz_2}
e^{-c\frac{4^jr^2}{t}}[w(2^jB_k)]^{\frac{1}{2}-\frac{1}{p_0}}\r.\noz\\
&&\hs\times\lf.\lf[\int_{U_j(B_k)}|f(x)|^{p_0}w(x)\,dx\r]^{\frac{1}{p_0}}\r\}^2\noz\\
&&\ls\sum_{k\in\ca}\lf\{\sum_{j=3}^\fz 2^{j\tz_1}\lf[\Upsilon\lf(\frac{2^jr}{\sqrt{t}}\r)\r]^{\tz_2}
e^{-c\frac{4^jr^2}{t}}[w(2^jB_k)]^{\frac12-\frac{1}{p_0}}\r.\noz\\
&&\hs\times\lf.\lf[\int_{U_j(B_k)}|f(x)|^{p}[w(x)]^{\frac{p}{2}}\,dx\r]^{\frac{1}{p}}\r.\noz\\
&&\hs\times\lf.\lf[\int_{2^jB_k}[w(x)]^{(1-\frac{p_0}{2})
\frac{p}{p-p_0}}\,dx\r]^{\frac{1}{p_0}-\frac{1}{p}}\r\}^2.\noz
\end{eqnarray}
Since $p\in (1,2)$, we see that $s=(1-\frac{p_0}{2})\frac{p}{p-p_0}\in(1,\fz)$.
From the fact $w\in RH_s(\rn)$ and the assumption $0\le t<r^2$, it follows that,
for all $k\in\ca$ and $j\geq 3$,
\begin{eqnarray}\label{eq 2.o1}
[w(2^jB_k)]^{\frac12-\frac{1}{p_0}}\lf[\int_{2^jB_k}[w(x)]^{(1-\frac{p_0}{2})
\frac{p}{p-p_0}}\r]^{\frac{1}{p_0}-\frac{1}{p}}
&&\ls |2^jB_k|^{\frac12-\frac{1}{p}}\\
&&\ls r^{n(\frac12-\frac1p)}\ls t^{\frac{n}{2}(\frac12-\frac1p)}.\noz
\end{eqnarray}
It is easy to see that there exist positive constants $c_1$ and $c_2$ such that
\begin{equation*}
2^{j\tz_1}\lf[\Upsilon\lf(\frac{2^jr}{\sqrt{t}}\r)\r]^{\tz_2}e^{-c\frac{4^jr^2}{t}}
\ls e^{-c_14^j}e^{-c_2\frac{r^2}{t}}.
\end{equation*}
By this, \eqref{eq 2.o1}, \eqref{eq 2.o2}, the H\"{o}lder inequality, $2/p>1$ and
the fact that $r=\frac{d(E,\,F)}{16}$, we know that
there exist positive constants $c$ and $\wz c$ such that
\begin{eqnarray}\label{eq 2.o3}
&&\|T_t(f)\|^2_{L^2(w,\,F)}\\
&&\hs\ls\lf[t^{-\frac{n}2(\frac1p-\frac12)}e^{-c_2\frac{r^2}{t}}\r]^2
\sum_{k\in\ca}\lf\{\sum_{j=3}^\fz e^{-c_14^j}
\lf[\int_{U_j(B_k)}|f(x)|^p[w(x)]^\frac p2\,dx\r]^{\frac1p}\r\}^2\noz\\
&&\hs\ls\lf[t^{-\frac{n}2(\frac1p-\frac12)}e^{-c_2\frac{r^2}{t}}\r]^2
\sum_{k\in\ca}\lf\{\sum_{j=3}^\fz e^{-\wz c4^j}
\int_{U_j(B_k)}|f(x)|^p[w(x)]^\frac p2\,dx\r\}^{\frac2p}\noz\\
&&\hs\ls\lf[t^{-\frac{n}2(\frac1p-\frac12)}e^{-c\frac{[d(E,\,F)]^2}{t}}\r]^2
\lf\{\sum_{k\in\ca}\sum_{j=3}^\fz e^{-\wz c4^j}
\int_{U_j(B_k)}|f(x)|^p[w(x)]^\frac p2\,dx\r\}^{\frac2p}\noz\\
&&\hs\ls\lf[t^{-\frac{n}2(\frac1p-\frac12)}e^{-c\frac{[d(E,\,F)]^2}{t}}\r]^2
\lf\{\int_E\sum_{j=3}^\fz\sum_{k\in\ca}e^{-\wz c4^j}
\chi_{U_j(B_k)}(x)|f(x)|^p [w(x)]^\frac p2\,dx\r\}^{\frac2p}.\noz
\end{eqnarray}
Notice that, for all $x\in\rn$, there exists some $k_0\in\nn$
such that $x\in B_{k_0}$. Then, we know that, for all $j\geq 3$,
\begin{eqnarray*}
\sum_{k\in\ca}\chi_{U_j(B_k)}(x)\le\sharp\{k\in\nn :\ x\in 2^{j}B_k\}
\le\sharp\{k\in\nn :\ x_k\in B(x_{k_0}, 2^{j+1}r)\}
\le 2^{n(j+3)},
\end{eqnarray*}
which further implies that there exists a positive constant $C$ such that,
for all $x\in\rn$,
\begin{eqnarray*}
\sum_{j=3}^\fz\sum_{k\in\ca}e^{-\wz c4^j}\chi_{U_j(B_k)}(x)
\le\sum_{j=3}^\fz 2^{n(j+3)}e^{-\wz c4^j}\le C<\fz.
\end{eqnarray*}
From this and \eqref{eq 2.o3}, we deduce that, for all $0\le t<[\frac{d(E,\,F)}{16}]^2$,
\begin{eqnarray}\label{eq 2.oo}
\|T_t(f)\|_{L^2(w,\,F)}\ls t^{-\frac{n}2(\frac1p-\frac12)}e^{-c\frac{[d(E,\,F)]^2}{t}}
\|f\|_{L^p(w^{\frac{p}{2}},\,E)}.
\end{eqnarray}

\emph{Case 2) $t\geq [\frac{d(E,\,F)}{16}]^2$}. In this case, let $r:=\sqrt{t}$.
We also choose a family of balls, $\{B_k\}_{k\in\nn}=\{B(x_k,r)\}_{k\in\nn}$, as in \emph{Case 1)},
where $k\in\nn$ and $x_k\in\rn$.
Let also
$$\mathcal{A}:=\{k\in\nn:\ F\cap B_k\neq\emptyset\}.$$
Then, by the Minkowski inequality,
we see that
\begin{eqnarray*}
\|T_t(f)\|^2_{L^2(w,\,F)}
&&\le\sum_{k\in\ca}\int_{B_k}\lf|T_t(f)(x)\r|^2w(x)\,dx\\
&&\le\sum_{k\in\ca}\lf\{\sum_{j=0}^\fz\lf[\int_{B_k}
\lf|T_t(\chi_{U_j(B_k)}f)(x)\r|^{2}w(x)\,dx\r]^{\frac{1}{2}}\r\}^{2}.
\end{eqnarray*}
For $j\in\{0,\,1,\,2\}$, we use \eqref{ODEB1} with $B=4B_k$ to bound it.
Then, by \eqref{ODEB3} and an argument similar to that used in the estimate of \emph{Case 1)}, we obtain
\begin{eqnarray*}
\|T_t(f)\|_{L^2(w,\,F)}\ls t^{-\frac{n}2(\frac1p-\frac12)}
\|f\|_{L^p(w^{\frac{p}{2}},\,E)}\ls
t^{-\frac{n}2(\frac1p-\frac12)}e^{-c\frac{[d(E,\,F)]^2}{t}}
\|f\|_{L^p(w^{\frac{p}{2}},\,E)}.
\end{eqnarray*}
This, together with \eqref{eq 2.oo}, then completes the proof of
Proposition \ref{pro equivalent}(i).

To show (ii), we use the same method as that used in the proof of (i).
In this case, we also first assume that $d(E,\,F)>0$ and $0\le t<[\frac{d(E,\,F)}{16}]^2$.
By the same argument as that used in \eqref{eq 2.o2}, we know that
\begin{eqnarray}\label{eq 2.ox}
\|T_t(f)\|^2_{L^2(w,\,F)}
&&\le\sum_{k\in\ca}\int_{B_k}\lf|T_t(f)(x)\r|^2w(x)\,dx\\
&&\le\sum_{k\in\ca}\lf\{\sum_{j=3}^\fz\lf[\int_{B_k}
\lf|T_t(\chi_{U_j(B_k)}f)(x)\r|^{2}w(x)\,dx\r]^{\frac{1}{2}}\r\}^{2}\noz\\
&&\ls\sum_{k\in\ca}\lf\{\sum_{j=3}^\fz 2^{j\tz_1}\lf[\Upsilon\lf(\frac{2^jr}{\sqrt{t}}\r)\r]^{\tz_2}
e^{-c\frac{4^jr^2}{t}}[w(2^jB_k)]^{\frac{1}{2}-\frac{1}{p_0}}\r.\noz\\
&&\hs\times\lf.\lf[\int_{U_j(B_k)}|f(x)|^{p_0}w(x)\,dx\r]^{\frac{1}{p_0}}\r\}^2\noz\\
&&\ls\sum_{k\in\ca}\lf\{\sum_{j=3}^\fz 2^{j\tz_1}\lf[\Upsilon\lf(\frac{2^jr}{\sqrt{t}}\r)\r]^{\tz_2}
e^{-c\frac{4^jr^2}{t}}[w(2^jB_k)]^{\frac12-\frac{1}{p_0}}\r.\noz\\
&&\hs\times\lf.\lf[\int_{U_j(B_k)}|f(x)|^{p_0}[w(x)]^{\frac{p_0}{2}}\,dx\r]^{\frac{1}{p_0}}
\|w\|_{L^\fz(2^jB_k)}^{\frac{1}{p_0}-\frac{1}{2}}\r\}^2.\noz
\end{eqnarray}
From $w\in RH_\fz(\rn)$ and $0\le t<r^2$, it follows that,
for all $k\in\mathcal{A}$ and $j\geq 3$,
\begin{eqnarray*}
\|w\|_{L^\fz(2^jB_k)}^{\frac{1}{p_0}-\frac12}[w(2^jB_k)]^{\frac12-\frac{1}{p_0}}
\ls |2^jB_k|^{\frac{1}{2}-\frac{1}{p_0}}\ls r^{n(\frac12-\frac{1}{p_0})}
\ls t^{\frac n2(\frac12-\frac{1}{p_0})}.
\end{eqnarray*}
By this, \eqref{eq 2.ox} and an argument similar to that used in the proof of \emph{Case 1)} of (i),
we see that, for all $0\le t<[\frac{d(E,\,F)}{16}]^2$,
\begin{eqnarray}\label{eq 2.ox1}
\|T_t(f)\|_{L^2(w,\,F)}\ls t^{-\frac{n}2(\frac1{p_0}-\frac12)}e^{-c\frac{[d(E,\,F)]^2}{t}}
\|f\|_{L^{p_0}(w^{\frac{p_0}{2}},\,E)}.
\end{eqnarray}

When $t\geq[\frac{d(E,\,F)}{16}]^2$, by an argument similar to that used in
the proof of \emph{Case 2)} of (i),
we know \eqref{eq 2.ox1} also holds true.
This finishes the proof of Proposition \ref{pro equivalent}(ii) and hence
Proposition \ref{pro equivalent}.
\end{proof}

\begin{rem}\label{rem off}
From \cite[Theorems 1 and 5]{CR14}, we deduce that,
if the matrix $A$ associated with $L_w$ (see \eqref{Lw}) is real symmetric,
then $\{e^{-tL_w}\}_{t\geq 0}$ and $\{tL_we^{-tL_w}\}_{t\geq 0}$
have heat kernels. Moreover, the heat kernels both satisfy the weighted
Gaussian bounds (see \cite[p.\,1\,(2)]{CR14}).
By \cite[Proposition 2.2]{am07ii}, we know that
the weighted Gaussian bounds (\cite[p.\,1\,(2)]{CR14}) is equivalent to the weighted $L^1$-$L^\fz$
off-diagonal estimates on balls.
Since $\mathcal{O}_{w}(L^1-L^\fz)\st \mathcal{O}_{w}(L^p-L^q)$ with $1\le p\le q\le \fz$
(see \cite[Comments 4]{am07ii}), if follows that, if the matrix $A$ is real symmetric,
then, for any $p_0\in[1,\,2)$, $tL_we^{-tL_w}\in \mathcal{O}_w(L^{p_0}-L^2)$.
From Proposition \ref{pro equivalent}, we further deduce that, for any $p\in[1,\,2)$,
$tL_we^{-tL_w}\in \mathcal{F}_w(L^p-L^2)$.
\end{rem}

Generally, we have the following conclusion.
\begin{prop}\label{pro off-diagonal}
For $n\geq 3$,
let $w^{-1}\in A_{2-\frac{2}{n}}(\rn)$, $p_-=\frac{2n}{n+2}$ and
$L_w$ be the degenerate elliptic operator satisfying \eqref{degenerate C1}
and \eqref{degenerate C2}. Then there exist positive constants $C$ and
$\wz C$ such that, for
all closed sets $E$ and $F$, $t\in(0,\,\fz)$ and $f\in L^{p_-}(w^{\frac {p_-} {2}},\,\rn)$
supported in $E$,
\begin{eqnarray}\label{eq off-diagonal}
\lf\|e^{-tL_w}(f)\r\|_{L^2(w,\,F)} \le C t^{-\frac{n}{2}(\frac{1}{p_-}-\frac{1}{2})}
e^{-\wz C\frac{[d(E,\,F)]^2}{t}}\|f\|_{L^{p_-}(w^{\frac{p_-}{2}},\,E)}.
\end{eqnarray}
\end{prop}

The proof of Proposition \ref{pro off-diagonal} relies on an exponential perturbation
method from \cite{Da95} and the boundedness of the Riesz potential in
weighted Lebesgue spaces from \cite{M-W74}.
We first introduce some notions and lemmas.

Let $\mathcal{E}(\rn)$ be the set of all bounded
real-valued functions $\phi\in C^\fz(\rn)$ such that, for all multi-indices
$\az\in (\zz_+)^n$ and $|\az|=1$, $\|\pa^\az \phi\|_{L^\fz(\rn)}\le 1$.
Now, for $\nu\in\rr_+:=(0,\fz)$ and $\phi\in\mathcal{E}(\rn)$,
let
\begin{eqnarray}\label{twist}
L_{\nu,\,\phi}:=e^{\nu\phi}L_we^{-\nu\phi}.
\end{eqnarray}
For all $f$, $g\in \mathcal{H}_0^1(w,\,\rn)$,
the \emph{twist sesquilinear form} $\mathfrak{a}_{\nu,\,\phi}$ is defined by setting
\begin{eqnarray*}
\mathfrak{a}_{\nu,\,\phi}(f,\,g):=\dint_{\rn}\lf[ A(x)\nabla (e^{-\nu \phi}f)(x)\r]
\cdot\nabla(e^{\nu\phi} g)(x)\,dx.
\end{eqnarray*}
Then, by the definition of $L_w$, we know that
\begin{equation}\label{eq sesqui-pertur}
\mathfrak{a}_{\nu,\,\phi}(f,\,g)
=\lf(L_{\nu,\,\phi}(f),\,g\r)_{L^2(w,\,\rn)}:=\int_\rn L_{\nu,\,\phi}(f)(x)\overline{g(x)}w(x)\,dx.
\end{equation}
Let $\{e^{-tL_{\nu,\,\phi}}\}_{t>0}$ be the heat semigroup
generated by $L_{\nu,\,\phi}$.

Notice that the conditions \eqref{degenerate C1} and \eqref{degenerate C2} imply that $L_w$ is of type
$\omega$, where $\omega:=\arctan(\Lambda/\lambda)\in [0,\,\frac{\pi}{2})$;
see \cite{M86} (also \cite[p.\,293]{CR08}) for the details.
Hence, for $z\in\Sigma(\pi/2-\omega)$, where
$$\Sigma(\pi/2-\omega):=\{z\in\cc\setminus\{0\}:\,|\arg z|<\pi/2-\omega\},$$
it holds true that
\begin{eqnarray*}
e^{-zL_w}(f)=\frac{1}{2\pi i}\int_{\bgz} e^{z\xi}
(\xi I+L_w)^{-1}(f)\,d\xi,
\end{eqnarray*}
where
$$\bgz:=\gz^+\cup\gz^-:=\lf\{z\in\cc:\
z=r^{i\tz},\,r\in(0,\,\fz)\r\}\bigcup \lf\{z\in\cc:\
z=r^{- i\tz},\,r\in(0,\,\fz)\r\}$$
for some $\tz\in(\pi/2+|\arg (z)|,\,\pi-\omega)$.
This, together with \eqref{twist}, implies that,
for all $t\in(0,\,\fz)$,
\begin{eqnarray}\label{equation}
e^{-tL_{\nu,\,\phi}}=e^{\nu\phi}e^{-tL_w}e^{-\nu\phi}.
\end{eqnarray}

The following two lemmas are, respectively,
\cite[Lemma 2.4]{ZCJY14} and \cite[Lemma 2.5]{ZCJY14}.

\begin{lem}[\cite{ZCJY14}]\label{lem perturbation}
Let $w\in A_2(\rn)$ and
$L_w$ be the degenerate elliptic operator satisfying
the degenerate elliptic conditions \eqref{degenerate C1}
and \eqref{degenerate C2}. Then there exists a positive constant $C$ such that,
for all $\nu\in\rr_+$, $\phi\in \mathcal{E}(\rn)$ and $f\in \mathcal{H}_0^1(w,\,\rn)$,
\begin{eqnarray}\label{eq perturbation}
\lf|\mathfrak{a}_{\nu,\,\phi}(f,\,f)-\mathfrak{a}(f,\,f)\r|\le \frac{1}{4}
\Re\lf\{\mathfrak{a}(f,\,f)\r\}+C\nu^2\|f\|_{L^2(w,\,\rn)}^2.
\end{eqnarray}
\end{lem}

\begin{lem}[\cite{ZCJY14}]\label{lem twist semi}
Let $w\in A_2(\rn)$, $k\in \zz_+$ and
$L_w$ be the degenerate elliptic operator satisfying
the degenerate elliptic conditions \eqref{degenerate C1}
and \eqref{degenerate C2}. Then there exist positive constants $C_0$ and $C_1$
such that, for all $\nu\in\rr_+$, $\phi\in\mathcal{E}(\rn)$,
$t\in(0,\,\fz)$ and $f\in L^2(w,\,\rn)$,
\begin{eqnarray}\label{eq twist semi}
\lf\|\lf(tL_{\nu,\,\phi}\r)^ke^{-tL_{\nu,\,\phi}}(f)\r\|_{L^2(w,\,\rn)}
\le C_0 e^{C_1\nu^2t}\|f\|_{L^2(w,\,\rn)}.
\end{eqnarray}
\end{lem}

Let $1<p<q<\fz$. Recall also the following definition of $A_{p,\,q}(\rn)$
weights from \cite{M-W74}.
A nonnegative and locally integrable function $w$ is said to belong to the
\emph{weight class} $A_{p,\,q}(\rn)$, if
\begin{eqnarray*}
\lf[w\r]_{A_{p,\,q}(\rn)}:=\dsup_{B\subset \rn}\lf\{\frac{1}{|B|}\dint_{B}
\lf[w(x)\r]^{q}\,dx\r\}^{\frac{1}{q}}\lf\{\frac{1}{|B|}\dint_{B}
\lf[w(x)\r]^{-p'}\,dx\r\}^{\frac{1}{p'}}<\fz,
\end{eqnarray*}
where the supremum is taken over all open balls $B\subset \rn$ and
$p':=\frac{p}{p-1}$ denotes the \emph{conjugate exponent} of $p$.

We are now in a position to prove Proposition \ref{pro off-diagonal}.
\begin{proof}[Proof of Proposition \ref{pro off-diagonal}]
Let $p_+:=\frac{2n}{n-2}$. It is easy to see that
$\frac{1}{p_-}+\frac{1}{p_+}=1$ and
\begin{eqnarray}\label{eq ppm}
\frac12-\frac1{p_+}=\frac1{p_-}-\frac12=\frac1n.
\end{eqnarray}
It is well known that $w\in A_{p,\,q}(\rn)$ if and only if $w^{-p'}\in A_{1+\frac{p'}{q}}(\rn)$
(see \cite[pp.\,266-267]{M-W74}),
where $1<p<q<\fz$.
Hence, $w^{-1}\in A_{2-\frac{2}{n}}(\rn)$ is equivalent to
$w^{1/2}\in A_{2,\,p_+}(\rn)$.
Then, by \cite[Theorem 4]{M-W74}, we know that
the Riesz potential
$$(-\Delta)^{-1/2}(f)(x):=\frac{1}{\gz(1)}\int_\rn \frac{f(y)}{|x-y|^{n-1}}\,dy,$$
where $x\in\rn$ and $\gz(1)=2\pi^{\frac n2}\Gamma(\frac12)/\Gamma(\frac{n-1}{2})$, is bounded from
$L^2(w,\,\rn)$ to $L^{p_+}(w^{\frac{p_+}{2}},\,\rn)$. Therefore,
for all $g\in L^2(w,\,\rn)$,
\begin{eqnarray*}
\lf\|\lf(-\bdz\r)^{-\frac{1}{2}}(g)\r\|_{L^{p_+}(w^{\frac{p_+}{2}},\,\rn)}\ls \lf\|g\r\|_{L^2(w,\,\rn)}.
\end{eqnarray*}
Moreover, since, for any $u\in C_c^\fz(\rn)$ and $x\in\rn$, it holds true that
\begin{eqnarray}\label{eq element-ineq}
|u(x)|\ls \int_{\rn}\frac{|\nabla u(y)|}{|x-y|^{n-1}}\,dy \ls \lf(-\bdz\r)^{-\frac{1}{2}}(|\nabla u|)(x)
\end{eqnarray}
(see \cite[p.\,125]{St70}),
this, combined with a density argument, implies that,
 for all $h\in \mathcal{H}_0^1(w,\,\rn)$,
\begin{eqnarray}\label{eq Sobolev}
\lf\|h\r\|_{L^{p_+}(w^{\frac{p_+}{2}},\,\rn)}\ls \lf\|\nabla h\r\|_{L^2(w,\,\rn)}.
\end{eqnarray}

Now, for all $t,\,\nu\in(0,\,\fz)$, $\phi\in\mathcal{E}(\rn)$ and $f\in L^2(w,\,\rn)$,
let $f_t:=e^{-tL_{\nu,\,\phi}}(f)$.
By \eqref{eq Sobolev} and the degenerate elliptic condition
\eqref{degenerate C2}, we obtain
\begin{eqnarray}\label{eq 2.x}
\lf\|f_t\r\|_{L^{p_+}(w^{\frac{p_+}{2}},\,\rn)}&&\ls \lf\|\nabla(f_t)\r\|_{L^2(w,\,\rn)}
\ls \lf[\Re \lf\{\mathfrak{a}(f_t,\,f_t)\r\}\r]^{1/2}.
\end{eqnarray}
From Lemma \ref{lem perturbation}, it follows that
\begin{eqnarray*}
\Re \lf\{\mathfrak{a}(f_t,\,f_t)\r\}
&&\le \lf|\Re \lf\{\mathfrak{a}(f_t,\,f_t)\r\}-\Re
\lf\{\mathfrak{a}_{\nu,\,\phi}(f_t,\,f_t)\r\}\r|+\lf|\Re
\lf\{\mathfrak{a}_{\nu,\,\phi}(f_t,\,f_t)\r\}\r|\\
&&\le\frac{1}{4}\Re \lf\{\mathfrak{a}(f_t,\,f_t)\r\}
+ C\nu^2\|f_t\|^2_{L^2(w,\,\rn)}
+\lf|\mathfrak{a}_{\nu,\,\phi}(f_t,\,f_t)\r|,
\end{eqnarray*}
where the positive constant $C$ is as in Lemma \ref{lem perturbation}.
This, together with \eqref{eq 2.x}, \eqref{eq sesqui-pertur},
the H\"older inequality and Lemma \ref{lem twist semi},
implies that there exists a positive constant $M_0$ such that
\begin{eqnarray}\label{eq 2.1}
\lf\|f_t\r\|_{L^{p_+}(w^{\frac{p_+}{2}},\,\rn)}&&\ls
\lf[\nu^2\lf\|f_t\r\|_{L^2(w,\,\rn)}^2+\lf|
\mathfrak{a}_{\nu,\,\phi}(f_t,\,f_t)\r|\r]^{1/2}\\
&&\nonumber\ls \lf[\nu^2\lf\|f_t\r\|_{L^2(w,\,\rn)}^2
+\lf\|L_{\nu,\,\phi}(f_t)\r\|_{L^2(w,\,\rn)}\lf\|f_t\r\|_{L^2(w,\,\rn)}\r]^{1/2}\\
&&\nonumber\ls \lf[\nu^2 e^{2 C_1\nu^2t}\lf\|f\r\|_{L^2(w,\,\rn)}^2+
\frac{1}{t}e^{2C_1\nu^2t} \lf\|f\r\|_{L^2(w,\,\rn)}^2\r]^{1/2}\\
&&\nonumber\ls t^{-1/2}e^{M_0 \nu^2 t}\|f\|_{L^2(w,\,\rn)},
\end{eqnarray}
where the positive constant $C_1$ is as in Lemma \ref{lem twist semi} and
the implicit positive constants are independent of $t$, $\nu$ and $f$.

Take $\phi\in{\mathcal E}(\rn)$ satisfying $\phi|_{E}\ge 0$
and $\phi|_{F}\le -\frac{d(E,\,F)}{1+\epsilon}$, where $\epsilon$ is some
suitable positive constant.
By  this, \eqref{equation} and
\eqref{eq 2.1}, we see that, for all $g\in L^2(w,\,F)$ supported in $F$,
\begin{eqnarray}\label{2.x1}
\quad\lf\|e^{-tL_w}\lf(g\r)\r\|_{L^{p_+}(w^{\frac{p_+}{2}},\,E)}=&&
 \lf\|e^{-\nu \phi}e^{\nu \phi} e^{-tL_w}\lf(e^{-\nu \phi}e^{\nu \phi}  g\r)\r\|_{L^{p_+}(w^{\frac{p_+}{2}},\,E)}\\
&&\le\lf\|e^{-tL_{\nu,\,\phi}}\lf(e^{\nu\phi}g\r)\r\|_{L^{p_+}(w^\frac{p_+}{2},\,E)}\noz\\
&&\ls t^{-1/2}e^{M_0 \nu^2 t}
\lf\|e^{\nu \phi} g\r\|_{L^2(w,\,F)}\noz\\
&&\ls t^{-1/2}
e^{M_0\nu^2 t}e^{-\nu\frac{d(E,\,F)}{1+\epsilon}}\|g\|_{L^2(w,\,F)},\noz
\end{eqnarray}
where the positive constant $M_0$ is as in \eqref{eq 2.1}.
This, combined with the choice that $\nu:=\frac{d(E,\,F)}{\wz{C_0}t}$
with $\wz{C_0}>(1+\epsilon)M_0$, implies that
there exists a positive constant $K_0$ such that,
for all $g\in L^2(w,\,F)$ supported in $F$,
\begin{eqnarray}\label{eq 2.2}
\lf\|e^{-tL_w}\lf(g\r)\r\|_{L^{p_+}(w^{\frac{p_+}{2}},\,E)}
&&\ls t^{-1/2} e^{-[\frac{1}{\wz{C_0}}(\frac{1}{1+\epsilon}-\frac{M_0}{\wz{C_0}})]
\frac{[d(E,\,F)]^2}{t}}\|g\|_{L^2(w,\,F)}\\
&&\sim t^{-1/2} e^{-K_0\frac{[d(E,\,F)]^2}{t}}\|g\|_{L^2(w,\,F)}.\noz
\end{eqnarray}

Using duality, the H\"older inequality,
\eqref{eq 2.2} and \eqref{eq ppm}, we conclude that,
for all $f\in L^{p_-}(w^{\frac{p_-}{2}},\,E)$ supported in $E$
and $g\in L^2(w,\,F)$ supported in $F$,
\begin{eqnarray*}
&&\lf|\dint_{F}e^{-tL_w^*}(f)(x)\overline{g(x)}w(x)\,dx\r|\\
&&\hs=\lf|\dint_{E}f(x) \overline{e^{-tL_w}(g)(x)}[w(x)]^{\frac{p_+}{2p_+}+\frac{1}{2}}\,dx\r|\\
&&\hs\le \lf\{\dint_{E}\lf|e^{-tL_w}(g)(x)\r|^{p_+}\lf[w(x)\r]^{\frac{p_+}{2}}\,dx\r\}^{1/p_+}
\lf\{\dint_{E}\lf|f(x)\r|^{p_-}\lf[w(x)\r]^{\frac{p_-}{2}}\,dx\r\}^{1/p_-}\\
&&\hs\ls t^{-\frac{n}{2}(\frac{1}{p_-}-\frac{1}{2})}e^{-K_0\frac{[d(E,\,F)]^2}{t}}\|g\|_{L^2(w,\,F)}
\|f\|_{L^{p_-}(w^{\frac{p_-}{2}},\,E)},
\end{eqnarray*}
where the positive constant $K_0$ is as in \eqref{eq 2.2}.
By this and the dual representation of the $L^2(w,\,F)$ norm of $e^{-tL_w^*}(f)$,
we see that
\begin{eqnarray*}
\lf\|e^{-tL_w^*}(f)\r\|_{L^2(w,\,F)}
\ls t^{-\frac{n}{2}(\frac{1}{p_-}-\frac{1}{2})}e^{-K_0\frac{[d(E,\,F)]^2}{t}}
\|f\|_{L^{p_-}(w^{\frac{p_-}{2}},\,E)}.
\end{eqnarray*}
Observing the above estimates also hold true via replacing $e^{-tL_w^\ast}$ by $e^{-tL_w}$,
we then complete the proof of Proposition \ref{pro off-diagonal}.
\end{proof}

\subsection{Weighted tent spaces}\label{s2.2}
\hskip\parindent
For all measurable functions $f$ on $\rr^{n+1}_+$ and $x\in\rn$,
let
\begin{equation*}
A(f)(x):=\lf[\iint_{\bgz(x)}|f(y,t)|^2\,\frac{dy\,dt}{t^{n+1}}\r]^{1/2},
\end{equation*}
where $\bgz(x)$ is as in \eqref{cone} with $\az=1$.
For all $p\in(0,\fz)$ and $w\in A_\fz(\rn)$, the \emph{weighted tent space}
$T^p_w(\rr_+^{n+1})$ is defined to be the space of all measurable functions $f$
such that $\|f\|_{T^p_w(\rr_+^{n+1})}:=\|A(f)\|_{L^p(w,\,\rn)}<\fz$.
When $w\equiv 1$, the space $T^p_w(\rr_+^{n+1})$ was studied in \cite{CMS85}
and is simply denoted by $T^p(\rr_+^{n+1})$.

For any open set $O\st\rn$, the tent over $O$ is defined by
$$\wh O:=\{(x,t)\in\rnn:\ \dist(x,\, O^\com)\geq t\}.$$

Let $p\in (0,1]$ and $w\in A_\fz(\rn)$. A measurable function $a$ on $\rr_+^{n+1}$
is called a \emph{$(w,\,p,\,\fz)$-atom} if there exists a ball $B\st\rn$ such that

(i) $\supp a\subset \widehat{B}$;

(ii) for all $q\in (1,\fz)$,
\begin{equation*}
\|a\|_{T^p(\rr_+^{n+1})}:=\lf\{\int_\rn\lf[\iint_{\bgz(x)}|a(y,t)|^2\dydt\r]^\frac{q}{2}\,dx\r\}^{\frac{1}{q}}
\le |B|^{\frac{1}{q}}[w(B)]^{-\frac{1}{p}}.
\end{equation*}

\begin{rem}\label{rem note}
(i) Every $\asize$-atom $a$ belongs to $T^p_w(\rr_+^{n+1})$ and $\|a\|_{T_w^p(\rr_+^{n+1})}\le C$,
where the positive constant $C$ is independent of $a$ (see \cite[p.\,7]{BCKYY13-1}).

(ii) If $\supp f\st \wh{B}$ for some ball $B\st\rn$, then $\supp A(f)\st B$.
\end{rem}

The following lemma is needed in the proof of Theorem \ref{pro second}.
\begin{lem}\label{lem element1}
Let $a$ be a $\asize$-atom, with $w\in A_\fz(\rn)$ and $p\in(0,1]$, and $\supp a\st \wh{B}$.
Then, for any $p_1\in (1,\fz)$, there exists a positive constant $C$, independent of $a$,
such that
\begin{equation*}
\|a\|_{T^{p_1}_w(\rnn)}\le C[w(B)]^{\frac{1}{p_1}-\frac{1}{p}}.
\end{equation*}
\end{lem}

\begin{proof}
Since $w\in A_\fz(\rn)$, by Lemma \ref{lem Ap-1}(ii), we know that
there exists some $r\in (1,\fz)$ such that $w\in RH_{r'}(\rn)$,
where $1/r+1/r'=1$.
From this, Remark \ref{rem note}(ii) and the H\"{o}lder inequality, it follows that
\begin{eqnarray*}
\|a\|_{T^{p_1}_w(\rnn)}
&&=\lf\{\int_\rn\lf[\iint_{\bgz(x)}|a(y,t)|^2\dydt\r]^\frac{p_1}{2} w(x)\,dx\r\}^{\frac{1}{p_1}}\\
&&\le\lf\{\int_B\lf[\iint_{\bgz(x)}|a(y,t)|^2\dydt\r]^\frac{rp_1}{2}\,dx\r\}^{\frac{1}{rp_1}}
\lf\{\int_B [w(x)]^{r'}\,dx\r\}^{\frac{1}{r'p_1}}\\
&&\ls |B|^{\frac{1}{rp_1}}[w(B)]^{-\frac{1}{p}}|B|^{\frac{1}{r'p_1}-\frac{1}{p_1}}
[w(B)]^{\frac{1}{p_1}}\ls [w(B)]^{\frac{1}{p_1}-\frac{1}{p}},
\end{eqnarray*}
which completes the proof of Lemma \ref{lem element1}.
\end{proof}

An important result concerning weighted tent spaces is that each function
in $T^p_w(\rnn)$ has an atomic decomposition. More precisely, we have the following
result, which is a slight variant of \cite[Theorem 2.6]{BCKYY13-1}.
\begin{lem}[\cite{BCKYY13-1}]\label{lem tent-decompositon}
Let $p\in (0,1]$, $w\in A_\fz(\rn)$ and $f\in T^p_w(\rnn)$.
Then there exist a sequence of $\asize$-atoms, $\{a_j\}_{j\in\nn}$,
and a sequence of numbers, $\{\lz_j\}_{j\in\nn}\st\cc$, such that
\begin{equation}\label{eq 2.0}
f=\sum_{j\in\nn}\lz_j a_j,
\end{equation}
where the series converges in $T^p_w(\rnn)$. Moreover, there exist positive constants
$\wz C$ and $C$, independent of $f$, such that
\begin{equation*}
\wz C\|f\|_{T^p_w(\rnn)}\le \lf\{\sum_{j\in\nn}|\lz_j|^p\r\}^{1/p}\le C\|f\|_{T^p_w(\rnn)}.
\end{equation*}
Furthermore, if $f\in T^p_w(\rnn)\cap T^2_w(\rnn)$, then the series in \eqref{eq 2.0}
converges in both $T^p_w(\rnn)$ and $T^2_w(\rnn)$.
\end{lem}

\begin{proof}
By \cite[Theorem 2.6]{BCKYY13-1}, we only need to show that
$\|f\|_{T^p_w(\rnn)}\ls \{\sum_{j\in\nn}|\lz_j|^p\}^{1/p}$ and the last conclusion
of this lemma, concerning the $T^2_w(\rnn)$ convergence of the series in \eqref{eq 2.0}.
For all $N\in\nn$, let
$$S_N:=\sum_{j=1}^N \lz_j a_j.$$
From Remark \ref{rem note}(i), it follows that $\{S_N\}_{N\in\nn}$
is a Cauchy sequence in $T^p_w(\rnn)$ and
$\|S_N\|_{T^p_w(\rnn)}\ls \{\sum_{j\in\nn}|\lz_j|^p\}^{1/p}$.
Since $S_N$ converges to $f$ in $T^p_w(\rnn)$ as $N\to\infty$, we find that
$\|f\|_{T^p_w(\rnn)}\ls \{\sum_{j\in\nn}|\lz_j|^p\}^{1/p}$.

Similar to the proof of \cite[Proposition 3.1]{JY10} (see also
the proof of \cite[Proposition 3.32]{HMM11}),
we further conclude that, if $f\in T^p_w(\rnn)\cap T^2_w(\rnn)$, then the series in \eqref{eq 2.0}
converges in both $T^p_w(\rnn)$ and $T^2_w(\rnn)$, which completes the proof of
Lemma \ref{lem tent-decompositon}.
\end{proof}

\section{Proof of Theorem \ref{pro second}}\label{s3}
\hskip\parindent
In this section, we give the proof of Theorem \ref{pro second}. To this end,
we first introduce some technical lemmas.

The following lemma is a well known result (see, for example, \cite[Theorem 2,\,p.\,87]{ST89}).
\begin{lem}\label{lem bound}
Let $\Phi\in \mathcal{S}(\rn)$ satisfy $\int_\rn \Phi(x)\,dx=0$,
$\Phi_t(x):=\frac{1}{t^n}\Phi(\frac{x}{t})$ for all $x\in\rn$ and $t\in(0,\fz)$,
and $w\in A_2(\rn)$.
The Littlewood-Paley $g$-function $g_\Phi$ and square function $S_\Phi$ are defined, respectively, by setting,
for all $f\in \mathcal{S}'(\rn)$ and $x\in\rn$,
$$g_\Phi(f)(x):=\lf[\int_0^\fz |f\ast\Phi_t(x)|^2\,\frac{dt}{t}\r]^{1/2}$$
and
$$S_\Phi(f)(x):=\lf[\iint_{\bgz(x)}\lf|f\ast\Phi_t(y)\r|^2\,\frac{dy\,dt}{t^{n+1}}\r]^{1/2}.$$
Then $g_\Phi$ and $S_\Phi$ are bounded on $L^2(w,\,\rn)$.
\end{lem}

\begin{rem}\label{rem dual}
Let $w\in A_2(\rn)$. Then it is easy to show that, via the pairing
$\la f,\,g\ra:=\int_\rn f(x)g(x)\,dx$,
where $f\in L^2(w,\,\rn)$ and $g\in L^2(w^{-1},\,\rn)$, $L^2(w^{-1},\,\rn)$
and the dual space of $L^2(w,\,\rn)$ coincide with equivalent norms.
\end{rem}

The following lemma plays a key role in the proof of Theorem \ref{pro second}.
\begin{lem}\label{lem element2}
Let $\Phi\in \mathcal{S}(\rn)$ satisfy $\int_\rn \Phi(x)\,dx=0$,
$\Phi_t(x)=\frac{1}{t^n}\Phi(\frac{x}{t})$ for all $x\in\rn$ and $t\in(0,\fz)$,
and $w\in A_2(\rn)$.
For any $a\in T^2_w(\rnn)$ and $x\in\rn$, let
\begin{equation}\label{eq pi}
\pi_\Phi(a)(x):=\int_0^\fz \lf(a(\cdot,t)\ast\Phi_t\r)(x)\,\frac{dt}{t}.
\end{equation}
Then $\pi_\Phi$ is bounded from $T^2_w(\rnn)$ to $L^2(w,\,\rn)$.
\end{lem}

\begin{proof}
Fix any $a\in T^2_w(\rnn)$ and let $\wz\Phi(x):=\Phi(-x)$ for all $x\in\rn$.
Then, for any $f\in L^2(w^{-1},\,\rn)$ with $\|f\|_{L^2(w^{-1},\,\rn)}=1$,
by the Fubini theorem, we see that
\begin{eqnarray}\label{eq 2.3}
\la\pi_\Phi(a),\,f\ra
&&=\int_\rn\pi_\Phi(a)(x)f(x)\,dx
=\int_0^\fz \int_\rn (a(\cdot,t)\ast\Phi_t)(x) f(x)\,dx\,\frac{dt}{t}\\
&&=\int_0^\fz\int_\rn a(y,t)\lf(\wz\Phi_t\ast f\r)(y)\,dy\,\frac{dt}{t}\noz\\
&&\sim \int_0^\fz\int_\rn\int_{B(y,t)}a(y,t)\lf(\wz\Phi_t\ast f\r)(y)\,\frac{dx}{t^n}\,dy\,\frac{dt}{t}\noz\\
&&\sim \int_\rn\iint_{\bgz(x)}a(y,t)\lf(\wz\Phi_t\ast f\r)(y)\,\frac{dy\,dt}{t^{n+1}}\,dx.\noz
\end{eqnarray}
Since $w\in A_2(\rn)$ is equivalent to $w^{-1}\in A_2(\rn)$, by \eqref{eq 2.3}, the
H\"{o}lder inequality and Lemma \ref{lem bound}, we find that,
for all $f\in L^2(w^{-1},\,\rn)$ with $\|f\|_{L^2(w^{-1},\,\rn)}=1$,
\begin{eqnarray*}
|\la \pi_\Phi,\,g\ra|
&&\ls \int_\rn\lf[\iint_{\bgz(x)}|a(y,t)|^2\,\frac{dy\,dt}{t^{n+1}}\r]^{\frac{1}{2}}
\lf[\iint_{\bgz(x)}\lf|\lf(\wz\Phi_t\ast f\r)(y)\r|^2\,\frac{dy\,dt}{t^{n+1}}\r]^{\frac{1}{2}}\,dx\\
&&\ls \|A(a)\|_{L^2(w,\,\rn)}\|g_{\wz\Phi}(f)\|_{L^2(w^{-1},\,\rn)}
\ls \|a\|_{T^2_w(\rnn)}.
\end{eqnarray*}
From Remark \ref{rem dual}, we further deduce that
$\|\pi(a)\|_{L^2(w,\,\rn)}\ls \|a\|_{T^2_w(\rnn)}$, which completes the proof
of Lemma \ref{lem element2}.
\end{proof}

By an argument similar to that used in the proof of \cite[Lemma 6]{KS08}, we see that
the following lemma holds true, the details being omitted.
\begin{lem}\label{lem element3}
Let $\vz\in\mathcal{S}(\rn)$. Then the condition $\int_\rn \vz(x)\,dx=0$
is equivalent to that there exist elements $\psi_k\in\mathcal{S}(\rn)$,
$k\in \{1,\,\ldots,\,n\}$, such that
$$\vz=\sum_{k=1}^n \partial_k \psi_k.$$
\end{lem}

To prove Theorem \ref{pro second}, we also need the following local
weighted Sobolev imbedding theorem (see \cite[Theorem (1.2)]{FKS82}).
\begin{lem}[\cite{FKS82}]\label{thm imbedding}
For any given $p\in (1,\fz)$ and $w\in A_p(\rn)$, there exist positive constants
$C$ and $\delta$ such that, for all balls $B\equiv B(x_B,r_B)$ of $\rn$
with $x_B\in\rn$ and $r_B\in (0,\fz)$, $u\in C^\fz_c(B)$, and
numbers $k_0\in (0,\fz)$ satisfying $1\le k_0\le \frac{n}{n-1}+\delta$,
\begin{equation*}
\lf[\frac{1}{w(B)}\int_B|u(x)|^{k_0p}w(x)\,dx\r]^{\frac1{k_0p}}\le Cr_B
\lf[\frac{1}{w(B)}\int_B|\nab u(x)|^{p}w(x)\,dx\r]^{\frac1p}.
\end{equation*}
\end{lem}

We are now in a position to prove Theorem \ref{pro second}.
\begin{proof}[Proof of Theorem \ref{pro second}]
Take $\vz\in C^\fz_c(B(0,1))$ satisfying $\int_0^\fz t|\xi|^2|\hat{\vz}(t\xi)|^2\,dt=1$
for all $\xi\in\rn\setminus\{0\}$ (for the existence of such functions, see \cite[Lemma 1.1]{FJW91}).
In what follows, for a function $\vz:\ \rn\to\rr$, $t\in(0,\fz)$ and $x\in\rn$,
let $\vz_t(x):=\frac{1}{t^n}\vz(\frac{x}{t})$.

Let $f\in\mathcal{H}_0^1(w,\,\rn)\cap H_w^{1,p}(\rn)$,
$\nab f=(\partial_1f,\,\ldots,\,\partial_nf)=:(g_1,\,\ldots,\,g_n)=:\vec{g}$
and, for all $(x,t)\in\rnn$, define
$$F(x,t):=t\div\lf(\vec g\ast \vz_t(x)\r)=\sum_{j=1}^n g_j\ast(\partial_j\vz)_t(x).$$
By \cite[Theorem 2,\,p.\,87]{ST89}, we know that, for all $j\in\{1,\,\ldots,\,n\}$,
$$\|S_{\partial_j \vz}(g_j)\|_{L^p(w,\,\rn)}\ls \|g_j\|_{H^p_w(\rn)},$$
where $p\in (0,1]$.
This further implies that, for every $j\in\{1,\,\ldots,\,n\}$,
$g_j\ast(\partial_j\vz)_t\in T^p_w(\rnn)$. Thus, $F\in T^p_w(\rnn)$
and $\|F\|_{T^p_w(\rnn)}\ls \|\nab f\|_{H_w^p(\rn)}$.

On the other hand, noticing that, for every $j\in\{1,\,\ldots,\,n\}$,
the square function $S_{\partial_j\vz}$ is bounded on $L^2(w,\,\rn)$ (Lemma \ref{lem bound})
and $g_j\in L^2(w,\,\rn)$, we have $S_{\partial_j\vz}(g_j)\in L^2(w,\,\rn)$,
which further implies $F\in T^2_w(\rnn)$.

Thus, $F\in T^p_w(\rnn)\cap T^2_w(\rnn)$. From Lemma \ref{lem tent-decompositon},
it follows that there exist a sequence of numbers, $\{\lz_k\}_{k\in\nn}\st\cc$, and
a sequence of $\asize$-atoms, $\{\az_k\}_{k\in\nn}$, such that
\begin{equation*}
F=\sum_{k=1}^\fz \lz_k \az_k\ \ \ \text{in}\ \ T^p_w(\rnn)\cap T^2_w(\rnn)
\end{equation*}
and
\begin{eqnarray*}
\lf\{\sum_{k=1}^\fz |\lz_k|^p\r\}^{\frac{1}{p}}\sim \|F\|_{T^p_w(\rnn)}\ls \|\nab f\|_{H_w^p(\rn)}.
\end{eqnarray*}
From Lemmas \ref{lem element2} and \ref{lem element1},
we deduce that, for every $j\in\{1,\,\ldots,\,n\}$,
\begin{equation}\label{eq 2.z1}
\pi_{\partial_j\vz}(F)=\sum_{k=1}^\fz\lz_k\pi_{\partial_j\vz}(\az_k)\ \ \ \text{in}\ \ L^2(w,\,\rn)
\end{equation}
and
\begin{eqnarray}\label{eq 2.y4}
\|\pi_{\partial_j\vz}(\az_k)\|_{L^2(w,\,\rn)}\le C \|\az_k\|_{T^2_w(\rnn)}\le C [w(B_k)]^{\frac12-\frac1p}.
\end{eqnarray}
where $\pi_{\partial_j\vz}$ is as in \eqref{eq pi} with $\Phi$ replaced by $\partial_j\vz$
and the positive constant $C$ is independent of $k$.

Since, for every $k\in\nn$, $\az_k$ is a $\asize$-atom, we know that
there exists some ball $B_k:=B(x_k,r_k)$ with $x_k\in\rn$ and $r_k\in (0,\fz)$ such that
$\supp \az_k \st \wh{B_k}$ and, for every $q\in (1,\fz)$,
\begin{equation}\label{eq 2.4}
\lf\{\int_\rn\lf[\iint_{\bgz(x)}|\az_k(y,t)|^2\,\frac{dy\,dt}{t^{n+1}}\r]^{\frac{q}{2}}\,dx\r\}^{\frac{1}{q}}
\le |B_k|^{\frac{1}{q}}[w(B_k)]^{-\frac{1}{p}}.
\end{equation}
For every $k\in\nn$ and $x\in\rn$, let $\bz_k(x):=-\int_0^\fz(\az_k(\cdot,t)\ast\vz_t)(x)\,dt$
and $\wz B_k:=\wz{c} B_k$, where the positive constant $\wz c\in (1,\fz)$, independent of $k$,
will be determined later.
Next, we prove that, for every $k\in\nn$,
\begin{equation*}
\bz_k\in \mathcal{H}_0^1(w,\,\wz B_k)
\end{equation*}
and
\begin{equation}\label{eq 2.z2}
\vec {b_k}:=\lf(\pi_{\pt_1\vz}(\az_k),\,\ldots,\,\pi_{\pt_n\vz}(\az_k)\r)=\nab\bz_k.
\end{equation}

Since $\supp \az_k\st \wh{B_k}$, it is easy to see $\supp \bz_k\st B_k$.
By the fact that $\az_k$ is a $\asize$-atom, the Minkowski integral inequality,
the Young inequality
and the H\"{o}lder inequality, we further know that
\begin{eqnarray}\label{eq 2.6}
\|\bz_k\|_{L^2(\rn)}
&&=\lf[\int_\rn\lf|\int_0^\fz(\az_k(\cdot,t)\ast\vz_t)(x)\,dt\r|^2\,dx\r]^{\frac{1}{2}}\\
&&\le\int_0^{r_k}\lf[\int_\rn|(\az_k(\cdot,t)\ast\vz_t)(x)|^2\,dx\r]^{\frac{1}{2}}\,dt\noz\\
&&\le\int_0^{r_k}\lf[\int_\rn|\az_k(x,t)|^2\,dx\r]^{\frac{1}{2}}
\lf[\int_\rn|\vz_t(x)|\,dx\r]\,dt\noz\\
&&\ls r_k\lf[\int_0^{r_k}\int_\rn|\az_k(x,t)|^2\,\frac{dx\,dt}{t}\r]^{\frac{1}{2}}<\fz,\noz
\end{eqnarray}
where the last inequality follows from \eqref{eq 2.4} with $q=2$.
Thus, $\bz_k\in L^2(\rn)$.

For every $k\in\nn$, $\dz\in(0,r_k)$ and $x\in\rn$, let
$F_{k,\dz}(x):=\int_\dz^\fz(\az_k(\cdot,t)\ast\vz_t)(x)\,dt$.
Then $\supp F_{k,\dz}\st B_k$.
From an argument similar to that used in the estimate \eqref{eq 2.6}, it follows that
$F_{k,\dz}\in L^2(\rn)$ and
\begin{eqnarray}\label{eq 2.x4}
\lim_{\dz\to 0}\|F_{k,\dz}-\bz_k\|_{L^2(\rn)}=0.
\end{eqnarray}
Next, we prove that,
for any $k\in\nn$, $\dz\in(0,r_k)$ and almost every $x\in\rn$,
the partial derivatives of $F_{k,\dz}$ exist.

For any $i\in\{1,\,\ldots,\,n\}$, let $\vec{e_i}:=(0,\,\ldots,\,0,\,1,\,0,\,\ldots,\,0)\in\rn$
be the $i^{th}$ standard coordinate vector and $h\in(0,\fz)$. Then, we see that,
for every $x\in\rn$, there exists some $\tz\in (0,1)$ such that
\begin{eqnarray}\label{eq 2.51}
&&\frac{F_{k,\dz}(x+h\vec{e_i})-F_{k,\dz}(x)}{h}\\
&&\hs=\frac{1}{h}\int_\dz^\fz t\int_\rn\az_k(y,t)
\frac{1}{t^n}\lf[\vz\lf(\frac{x+h\vec{e_i}-y}{t}\r)-\vz\lf(\frac{x-y}{t}\r)\r]\,dy\,\frac{dt}{t}\noz\\
&&\hs=\int_\dz^\fz\int_\rn\az_k(y,t)\frac{1}{t^n}(\pt_i\vz)
\lf(\frac{x+\tz h\vec{e_i}-y}{t}\r)\,dy\,\frac{dt}{t}.\noz
\end{eqnarray}
Since $\vz\in C_c^\fz(B(0,1))$, it follows that, when $0<h<\dz$, there
exists a positive constant $C_{(\vz)}$, depending on $\vz$, such that
\begin{eqnarray}\label{eq 2.52}
&&\lf|\int_\rn\az_k(y,t)\frac{1}{t^n}(\pt_i\vz)\lf(\frac{x+\tz h\vec{e_i}-y}{t}\r)\,dy\r|\\
&&\hs\le C_{(\vz)}\int_\rn|\az_k(y,t)|\frac{1}{t^n}\chi_{B(0,2)}\lf(\frac{x-y}{t}\r)\,dy\noz\\
&&\hs=C_{(\vz)}\lf(|\az_k(\cdot,t)|\ast\lf(\chi_{B(0,2)}\r)_t\r)(x)=:G(x,t).\noz
\end{eqnarray}
By the Minkowski integral inequality, the Young inequality,
the H\"{o}lder inequality, \eqref{eq 2.4} with $q=2$
and the fact that $\az_k$ is a $\asize$-atom, we conclude that
\begin{eqnarray*}
\lf\{\int_\rn\lf[\int_\dz^\fz|G(x,t)|\,\frac{dt}{t}\r]^2\,dx\r\}^{\frac{1}{2}}
&&\le\int_\dz^\fz\lf[\int_\rn|G(x,t)|^2\,dx\r]^{\frac12}\,\frac{dt}{t}\\
&&\le C_{(\vz)}\int_\dz^{r_k}\|\az_k(\cdot,t)\|_{L^2(\rn)}
\lf\|\lf(\chi_{B(0,2)}\r)_t\r\|_{L^1(\rn)}\,\frac{dt}{t}\\
&&\le C_{(\vz,r_k,\dz)}\lf[\int_0^{r_k}\int_\rn|\az_k(x,t)|^2\,\frac{dx\,dt}{t}\r]^{\frac12}<\fz,
\end{eqnarray*}
which implies that, for almost every $x\in\rn$, $\int_\dz^\fz|G(x,t)|\,\frac{dt}{t}<\fz$.
By this, \eqref{eq 2.51}, \eqref{eq 2.52} and the dominated convergence theorem,
we find that, for almost every $x\in\rn$,
\begin{eqnarray*}
\pt_i F_{k,\dz}(x)=\int_\dz^\fz\lf(\az_k(\cdot,t)\ast(\pt_i\vz)_t\r)(x)\,\frac{dt}{t}.
\end{eqnarray*}
Moreover, by a simple calculation, we further see that $\pt_i F_{k,\dz}$, $i\in\{1,\,\ldots,\,n\}$,
is just the weak derivative of $F_{k,\,\dz}$.
From an argument similar to that used in the proof of Lemma \ref{lem element2}, we conclude that,
for every $k\in\nn$, $\dz\in (0,r_k)$ and $i\in\{1,\,\ldots,\,n\}$,
$\pt_i F_{k,\dz}\in L^2(w,\,\rn)$ and
\begin{eqnarray}\label{eq 2.x1}
\lim_{\dz\to 0}\|\nab F_{k,\dz}-\vec{b_k}\|_{L^2(w,\,\rn)}=0.
\end{eqnarray}

Take $\phi\in C^\fz_c(B(0,1))$ satisfying $\int_\rn \phi(x)\,dx=1$ and let
$\phi_\vez(x):=\frac{1}{\vez^n}\phi(\frac{x}{\vez})$ for all $x\in\rn$ and $\vez\in(0,\fz)$.
Since, for all $k,\,n\in\nn$, $\supp F_{k,1/n}\st B_k$, $F_{k,1/n}\in L^2(\rn)$
and $\nab F_{k,1/n}\in L^2(w,\,\rn)$,
from \cite[p.\,123]{St70}, \cite[Theorem 2.1]{Du01} and \cite[Theorem 2.1.4]{Tu00},
it follows that there exist a sequence $\{\vez_n\}_{n\in\nn}$ of positive numbers
satisfying $\lim_{n\to\fz}\vez_n=0$ and a positive constant $\wz c\in(1,\fz)$ such that
\begin{eqnarray*}
\wz F_{k,1/n}:=F_{k,1/n}\ast\phi_{\vez_n}\in C^\fz_c(\wz B_k)\ \ \text{with}\ \wz{B}_k=\wz c B_k, \ \
\nab\wz F_{k,1/n}= (\nab F_{k,1/n})\ast\phi_{\vez_n},
\end{eqnarray*}
\begin{eqnarray}\label{eq 2.x2}
\lf\|\wz F_{k,1/n}-F_{k,1/n}\r\|_{L^2(\rn)}< 2^{-n}
\end{eqnarray}
and, for $w\in A_2(\rn)$,
\begin{eqnarray}\label{eq 2.x3}
\lf\|\nab \wz F_{k,1/n}-\nab F_{k,1/n}\r\|_{L^2(w,\,\rn)}<2^{-n}.
\end{eqnarray}
From \eqref{eq 2.x3} and \eqref{eq 2.x1}, we deduce that
\begin{eqnarray}\label{eq 2.y1}
\lim_{n\to\fz}\lf\|\nab \wz F_{k,\vez_n}-\vec{b_k}\r\|_{L^2(w,\,\rn)}= 0.
\end{eqnarray}
By this, the fact that $\wz F_{k,1/n}\in C^\fz_c(\wz B_k)$ and Lemma \ref{thm imbedding},
we know that $\{\wz F_{k,1/n}\}_{n\in\nn}$ is a Cauchy sequence in $L^2(w,\,\wz B_k)$
and
\begin{eqnarray}\label{eq 2.y2}
\lf\|\wz F_{k,1/n}\r\|_{L^2(w,\,\wz B_k)}\le C r_k\lf\|\nab \wz F_{k,1/n}\r\|_{L^2(w,\,\wz B_k)},
\end{eqnarray}
where the positive constant $C$ is independent of $k$ or $n$.
Therefore, there exists a function $g_k\in L^2(w,\,\rn)$ such that
\begin{eqnarray}\label{eq 2.y3}
\lim_{n\to\fz}\lf\|\wz F_{k,\vez_n}-g_k\r\|_{L^2(w,\,\rn)}=0,
\end{eqnarray}
which further implies that there exists a subsequence of $\{\wz F_{k,\vez_n}\}_{n\in\nn}$
(without loss of generality, we use the same notation as the original sequence) such that,
for almost every $x\in\rn$, $\lim_{n\to\fz}\wz F_{k,\vez_n}(x)=g_k(x)$.

On the other hand, by \eqref{eq 2.x2} and \eqref{eq 2.x4}, we see that
$\lim_{n\to\fz}\|\wz F_{k,\vez_n}-\bz_k\|_{L^2(\rn)}=0$, which further implies
that there exists a subsequence of $\{\wz F_{k,\vez_n}\}_{n\in\nn}$
(without loss of generality, we use the same notation as the original sequence again) such that,
for almost every $x\in\rn$, $\lim_{n\to\fz}\wz F_{k,\vez_n}(x)=\bz_k(x)$.

Therefore, for every $k\in\nn$, $\bz_k=g_k\in L^2(w,\,\rn)$.
From this, \eqref{eq 2.y1}, \eqref{eq 2.y2} and \eqref{eq 2.y3}, we deduce that
\begin{eqnarray*}
\lim_{n\to\fz}\|\wz F_{k,\vez_n}-\bz_k\|_{\mathcal{H}^1_0(w,\,\wz B_k)}=0,\ \
\bz_k\in \mathcal{H}_0^1(w,\,\wz B_k),\ \  \nab \bz_k=\vec{b_k}
\end{eqnarray*}
and
\begin{eqnarray*}
\|\bz_k\|_{L^2(w,\,\wz B_k)}\ls r_k\|\nab \bz_k\|_{L^2(w,\,\wz B_k)}.
\end{eqnarray*}
This, together with \eqref{eq 2.y4},
further implies that, for every $k\in\nn$, $\bz_k$ is an $H^{1,p}_w(\rn)$-atom
associated to the ball $\wz B_k$ up to a harmless positive constant
independent of $k$.

By \eqref{eq 2.z1} and \eqref{eq 2.z2}, we see that
\begin{eqnarray}\label{eq 2.w1}
\sum_{k=1}^\fz \lz_k \nab \bz_k=\sum_{k=1}^\fz \lz_k \vec{b_k}=-\int_0^\fz \nab \lf(F\ast\vz_t\r)\,dt\ \ \
\text{in}\ \ L^2(w,\,\rn).
\end{eqnarray}
Next, we prove
\begin{eqnarray}\label{eq 2.w2}
-\int_0^\fz \nab \lf(F\ast\vz_t\r)\,dt=\vec{g}=\nab f\ \ \ \text{in}\ \ L^2(w,\,\rn).
\end{eqnarray}
Since, for any $u\in C_c^\fz(\rn)$ and all $\xi\in\rn$, it holds true that
\begin{eqnarray*}
&&\lf\{-\int_0^\fz\lf[t\div\lf((\nab u)\ast\vz_t\r)\r]\vz_t\,dt\r\}^\wedge(\xi)\\
&&\hs=-\int_0^\fz\lf\{\lf[t\div\lf((\nab u)\ast\vz_t\r)\r]\ast\vz_t\r\}^\wedge(\xi)\,dt\\
&&\hs=-\int_0^\fz\lf\{t\sum_{j=1}^n\pt_j\lf((\pt_j u)\ast\vz_t\r)\r\}^\wedge(\xi)\wh \vz(t\xi)\,dt\\
&&\hs=-i\int_0^\fz t\sum_{j=1}^n\xi_j\lf((\pt_ju)\ast\vz_t\r)^{\wedge}(\xi)\wh\vz(t\xi)\,dt
=\int_0^\fz t\sum_{j=1}^n[\xi_j\wh \vz(t\xi)]^2\wh u(\xi)\,dt\\
&&\hs=\wh u(\xi)\int_0^\fz t[|\xi|\wh \vz(t\xi)]^2\,dt=\wh u(\xi),
\end{eqnarray*}
then, it follows that $-\int_0^\fz\lf[t\div\lf((\nab u)\ast\vz_t\r)\r]\vz_t\,dt=u$
and hence
$$-\int_0^\fz\nab\lf[t\div\lf((\nab u)\ast\vz_t\r)\r]\vz_t\,dt=\nab u,$$
which, together with
$f\in \mathcal{H}_0^1(w,\,\rn)$ and a density argument, implies that \eqref{eq 2.w2} holds true.

Hence, from \eqref{eq 2.w1} and \eqref{eq 2.w2}, it follows that
\begin{eqnarray}\label{eq 2.xx1}
\nab f=\sum_{k=1}^\fz \lz_k \nab \bz_k\ \ \ \text{in}\ \ L^2(w,\,\rn),
\end{eqnarray}
where $\{\bz_k\}_{k\in\nn}$ is a sequence of $H^{1,p}_w(\rn)$-atoms up to a harmless
positive constant.

To complete the proof of Theorem \ref{pro second}, we still need to show \eqref{eq 2.xx}.
From \eqref{eq 2.xx1}, it is easy to see
$\nab f=\sum_{k=1}^\fz \lz_k \nab \bz_k$ in $\mathcal{S}'(\rn)$, which further implies
that, for any $\eta\in \mathcal{S}(\rn)$,
\begin{eqnarray*}
\int_\rn f(x)\nab \eta(x)\,dx=\sum_{k=1}^\fz \lz_k\int_\rn \bz_k(x)\nab \eta(x)\,dx.
\end{eqnarray*}
Then, by Lemma \ref{lem element3}, we obtain \eqref{eq 2.xx}, which
completes the proof of Theorem \ref{pro second}.
\end{proof}

\section{Proof of Theorem \ref{thm main}}\label{s4}

\begin{proof}[Proof of Theorem \ref{thm main}]
By \cite[Theorem 1.6]{ZCJY14}, we know that,
for $p\in (\frac{n}{n+1},1]$, $w\in A_{q_0}(\rn)$ with $q_0\in [1,\frac{p(n+1)}{n})$,
and $f\in H_{L_w}^p(\rn)\cap L^2(w,\,\rn)$,
\begin{eqnarray}\label{eq 4.x0}
\|f\|_{H_{L_w,\,{\rm Riesz}}^p(\rn)}=\|\nab L_w^{-1/2}f\|_{H^p_w(\rn)}\ls \|f\|_{H_{L_w}^p(\rn)},
\end{eqnarray}
which implies that
\begin{eqnarray}\label{eq 4.x1}
\lf(H_{L_w}^p(\rn)\cap L^2(w,\,\rn)\r)\st \lf(H_{L_w,\,{\rm Riesz}}^p(\rn)\cap L^2(w,\,\rn)\r).
\end{eqnarray}

Next, we prove the reverse inclusion. To this end, we only need
to show that, for any $h\in \rize\cap L^2(w,\,\rn)$,
\begin{eqnarray}\label{eq 3.0}
\|h\|_{H_{L_w}^p(\rn)}\ls \|\nab L_w^{-1/2}h\|_{H_w^p(\rn)}.
\end{eqnarray}

Let $f:=L_w^{-1/2}h$. Then, by \cite[p.\,281,\,Theorem 3.35]{Ka95} and \cite[Theorem 1.1]{CR13}, we see that
$f\in \mathcal{H}_0^1(w,\,\rn)$ and $\|\nab f\|_{L^2(w,\,\rn)}\sim \|L_w^{1/2}f\|_{L^2(w,\,\rn)}$.
For any $x\in\rn$, let
\begin{eqnarray*}
S_1 (h)(x):=
\lf[\iint_{\bgz(x)}\lf|t\sqrt{L_w}e^{-t^2L_w}h(y)\r|^2w(y)\,\frac{dy}{w(B(x,t))}\,\frac{dt}{t}\r]^{1/2}.
\end{eqnarray*}
Similar to proofs of \cite[Proposition 4.9 and Corollary 4.17]{HMM11},
we conclude that
\begin{eqnarray*}%\label{eq 3.1}
\|S_1 (h)\|_{L^p(w,\,\rn)}\sim \|h\|_{H_{L_w}^p(\rn)}.
\end{eqnarray*}
Therefore, to prove \eqref{eq 3.0}, it suffices to show
\begin{eqnarray}\label{eq 3.2}
\lf\|S_1\lf(\sqrt{L_w}(f)\r)\r\|_{L^p(w,\,\rn)}\ls \|\nab f\|_{H_w^p(\rn)}.
\end{eqnarray}
Since $f\in \mathcal{H}_0^1(w,\,\rn)$ and $\nab f\in H_w^p(\rn)$,
we see that $f\in \mathcal{H}_0^1(w,\,\rn)\cap H^{1,p}_w(\rn)$. Then, by Theorem \ref{pro second},
there exist a sequence of numbers, $\{\lz_k\}_{k\in\nn}\st\cc$,
and a sequence of $H^{1,p}_w(\rn)$-atoms, $\{\bz_k\}_{k\in\nn}$, such that
\begin{equation}\label{eq 3.4}
\nab f=\sum_{k=1}^\fz \lz_k \nab \bz_k\ \ \ \ \text{in}\ L^2(w,\,\rn)
\end{equation}
and
\begin{equation*}
\lf\{\sum_{k=1}^\fz |\lz_k|^p\r\}^{\frac{1}{p}}\ls \|\nab f\|_{H_w^p(\rn)}.
\end{equation*}

We claim that, to prove \eqref{eq 3.2}, it suffices to show that, for any $H^{1,p}_w(\rn)$-atom $a$,
there exists a positive constant $C$, independent of $a$, such that
\begin{equation}\label{eq 3.3}
\lf\|S_1\lf(\sqrt{L_w}(a)\r)\r\|_{L^p(w,\,\rn)}\le C.
\end{equation}
Indeed, since $L_w$ has a bounded $H_\fz$ calculus in $L^2(w,\,\rn)$
(see \cite{CR08} and \cite{M86}),
by \cite[p.\,487]{BL11} (see also \cite{M86,JY11}), we know that
$S_1$ is bounded on $L^2(w,\,\rn)$.
Therefore, by this and \cite[Theorem 1.1]{CR13}, we have
\begin{eqnarray}\label{eq 3.5}
\lf\|S_1\lf(\sqrt{L_w}(f)\r)\r\|_{L^2(w,\,\rn)}\ls\lf\|\sqrt{L_w}(f)\r\|_{L^2(w,\,\rn)}
\sim \|\nab f\|_{L^2(w,\,\rn)}.
\end{eqnarray}
From \eqref{eq 3.5} and \eqref{eq 3.4}, we deduce that
\begin{eqnarray*}
\lim_{N\to\fz}
\lf\|S_1\lf(\sqrt{L_w}(f)\r)-S_1\lf(\sqrt{L_w}\lf(\sum_{k=1}^N\lz_k\bz_k\r)\r)\r\|_{L^2(w,\,\rn)}=0.
\end{eqnarray*}
Hence, there exists a subsequence of $\{S_1(\sqrt{L_w}(\sum_{k=1}^N\lz_k\bz_k))\}_{N=1}^\fz$
(without loss of generality, we use the same notation as the original sequence) such that,
for almost every $x\in\rn$,
\begin{eqnarray*}
\lim_{N\to\fz}S_1\lf(\sqrt{L_w}\lf(\sum_{k=1}^N\lz_k\bz_k\r)\r)(x)=S_1\lf(\sqrt{L_w}(f)\r)(x).
\end{eqnarray*}
From this and the Minkowski inequality, we deduce that, for almost every $x\in\rn$,
$$S_1\lf(\sqrt{L_w}(f)\r)(x)\le \sum_{k=1}^\fz |\lz_k|S_1\lf(\sqrt{L_w}(\bz_k)\r)(x).$$
By this and \eqref{eq 3.3}, we know that
\begin{eqnarray*}
\lf\|S_1\lf(\sqrt{L_w}(f)\r)\r\|_{L^p(w,\,\rn)}
&&\le \lf[\sum_{k=1}^\fz\int_\rn |\lz_k|^p\lf[S_1\lf(\sqrt{L_w}(\bz_k)\r)(x)\r]^p w(x)\,dx\r]^{1/p}\\
&&\ls \lf[\sum_{k=1}^\fz|\lz_k|^p\r]^{1/p}\ls\|\nab f\|_{H_w^p(\rn)}.
\end{eqnarray*}
Thus, \eqref{eq 3.3} implies \eqref{eq 3.2}.

It remains to prove \eqref{eq 3.3}. Let $a$ be an $H^{1,p}_w(\rn)$-atom associated to
a ball $B:=(x_B,r_B)$ with $x_B\in\rn$ and $r_B\in (0,\fz)$.
Then, by the H\"{o}lder inequality, we find that
\begin{eqnarray}\label{eq 3.xx}
\hs\hs\hs&&\lf\|S_1\lf(\sqrt{L_w}(a)\r)\r\|^p_{L^p(w,\,\rn)}\\
&&\hs=\int_\rn \lf[S_1\lf(\sqrt{L_w}(a)\r)(x)\r]^p w(x)\,dx\noz\\
&&\hs=\sum_{j=0}^\fz \int_{\ujb}
\lf[\iint_{\bgz(x)}\lf|tL_w e^{-t^2L_w}(a)(y)\r|^2w(y)\frac{dy}{w(B(x,t))}\,\frac{dt}{t}\r]^{\frac{p}{2}}
w(x)\,dx\noz\\
&&\hs\le \sum_{j=0}^\fz \lf[w(2^jB)\r]^{1-\frac{p}{2}}\lf[\int_{\ujb}
\iint_{\bgz(x)}\lf|tL_we^{-t^2L_w}(a)(y)\r|^2
\frac{w(y)\,dy}{w(B(x,t))}\,\frac{dt}{t}w(x)\,dx\r]^{\frac{p}{2}}\noz\\
&&\hs\ls\sum_{j=3}^\fz \lf[w(2^jB)\r]^{1-\frac{p}{2}}\lf[\iint_{R(\ujb)}
\lf|tL_we^{-t^2L_w}(a)(y)\r|^2
w(y)\,dy\,\frac{dt}{t}\r]^{\frac{p}{2}}\noz\\
&&\hs\hs+[w(B)]^{1-\frac{p}{2}}\lf\|S_1\lf(\sqrt{L_w}(a)\r)\r\|_{L^2(w,\,\rn)}^p\noz\\
&&\hs\ls \sum_{j=3}^\fz \lf[w(2^jB)\r]^{1-\frac{p}{2}}\lf[\int_0^\fz\int_{(2^{j-2}B)^\com}
\lf|t^2L_we^{-t^2L_w}(a)(y)\r|^2 w(y)\,dy\,\frac{dt}{t^3}\r]^{\frac{p}{2}}\noz\\
&&\hs\hs+ \sum_{j=3}^\fz \lf[w(2^jB)\r]^{1-\frac{p}{2}}\lf[\int_{2^{j-2}r_B}^\fz\int_{2^{j-2}B}
\lf|t^2L_we^{-t^2L_w}(a)(y)\r|^2 w(y)\,dy\,\frac{dt}{t^3}\r]^{\frac{p}{2}}\noz\\
&&\hs\hs +[w(B)]^{1-\frac{p}{2}}\lf\|S_1\lf(\sqrt{L_w}(a)\r)\r\|_{L^2(w,\,\rn)}^p
=:{\rm I}+{\rm II}+{\rm III},\noz
\end{eqnarray}
where $\ujb$ is as in \eqref{eq-def of ujb} and $R(\ujb)$ is as in \eqref{eq tent}
with $F$ replaced by $\ujb$.

From \eqref{eq 3.5} and the fact that $a$ is an $H^{1,p}_w(\rn)$ atom,
it follows that
\begin{eqnarray}\label{eq 3.x1}
{\rm III}\ls [w(B)]^{1-\frac p2}\lf\|\nab a\r\|_{L^2(w,\,\rn)}^p\ls 1.
\end{eqnarray}

For ${\rm II}$, by the assumption that $\{tL_w e^{-tL_w}\}_{t\geq 0}$ satisfies
$L^r-L^2$ weighted full off-diagonal estimates with $r\in (1,2)$,
$w\in A_q(\rn)$ with $q\in [1,\frac{2p}{2-p}(\frac1r-\frac12+\frac1n))$,
and Lemma \ref{lem Ap-2}, we conclude that
\begin{eqnarray}\label{eq 3.6}
{\rm II}
&&= \sum_{j=3}^\fz \lf[w(2^jB)\r]^{1-\frac{p}{2}}\lf[\int_{2^{j-2}r_B}^\fz\int_{2^{j-2}B}
\lf|t^2L_we^{-t^2L_w}(a)(y)\r|^2 w(y)\,dy\,\frac{dt}{t^3}\r]^{\frac{p}{2}}\\
&&\ls \sum_{j=3}^\fz \lf[w(2^jB)\r]^{1-\frac{p}{2}}\lf\{\int_{2^{j-2}r_B}^\fz
t^{-2n\lf(\frac{1}{r}-\frac{1}{2}\r)}
\lf[\int_B |a(y)|^r[w(y)]^{\frac{r}{2}}\,dy\r]^{\frac{2}{r}}\,\frac{dt}{t^3}\r\}
^{\frac{p}{2}}\noz\\
&&\ls \sum_{j=3}^\fz\lf[2^{qnj}w(B)\r]^{1-\frac p2}\lf(2^j r_B\r)^{-p\lf(1+\frac{n}{r}-\frac{n}{2}\r)}
\lf[\int_B |a(y)|^{r}[w(y)]^{\frac{r}{2}}\,dy\r]^{\frac pr}.\noz
\end{eqnarray}
From the H\"{o}lder inequality, the fact that $a$ is an $H^{1,p}_w(\rn)$-atom
and the weighted Sobolev inequality \eqref{eq Sobolev} with $p_+=\frac{2n}{n-2}$, it follows that
\begin{eqnarray}\label{eq 3.6x}
&&\lf[\int_B |a(y)|^{r}[w(y)]^{\frac{r}{2}}\,dy\r]^{\frac p{r}}\\
&&\hs\le \lf[\int_B |a(y)|^{p_+}[w(y)]^{\frac{p_+}{2}}\,dy\r]^{\frac{p}{p_+}}
|B|^{p(\frac{1}{r}-\frac{1}{p_+})}\noz\\
&&\hs\ls \lf[\int_B |\nab a(y)|^2w(y)\,dy\r]^{\frac{p}{2}}|B|^{p\lf(\frac{1}{r}-\frac12+\frac1n\r)}
\ls [w(B)]^{\frac{p}{2}-1}(r_B)^{p\lf(\frac{n}{r}-\frac n2+1\r)}.\noz
\end{eqnarray}
Combining this, \eqref{eq 3.6} and the fact that
$1\le q<\frac{2p}{2-p}(\frac1r-\frac12+\frac1n)$, we conclude that
\begin{equation}\label{eq 3.x2}
{\rm II}\ls \sum_{j=3}^\fz 2^{qnj(1-\frac{p}{2})}2^{-jp\lf(\frac{n}{r}-\frac n2+1\r)}\ls 1.
\end{equation}

For ${\rm I}$, we write
\begin{eqnarray}\label{eq 3.x3}
{\rm I}
&&\le \sum_{j=3}^\fz \lf[w(2^jB)\r]^{1-\frac{p}{2}}\lf[\int_0^{2^jr_B}\int_{(2^{j-2}B)^\com}
\lf|t^2L_we^{-t^2L_w}(a)(y)\r|^2 w(y)\,dy\,\frac{dt}{t^3}\r]^{\frac{p}{2}}\\
&&\hs +\sum_{j=3}^\fz \lf[w(2^jB)\r]^{1-\frac{p}{2}}\lf[\int_{2^jr_B}^\fz\int_{(2^{j-2}B)^\com}
\cdots\r]^{\frac{p}{2}}=:{\rm I}_1+{\rm I}_2.\noz
\end{eqnarray}
By Proposition \ref{pro off-diagonal},
the fact that $w\in A_q(\rn)$ with $1\le q<\frac{2p}{2-p}(\frac1r-\frac12+\frac1n)$,
Lemma \ref{lem Ap-2} and \eqref{eq 3.6x}, we find that
\begin{eqnarray}\label{eq 3.x31}
\hs\hs\hs{\rm I}_1
&&\ls \sum_{j=3}^\fz \lf[w(2^jB)\r]^{1-\frac{p}{2}}\lf[\int_0^{2^jr_B}
t^{-2n(\frac{1}{r}-\frac12)}e^{-\frac{(2^jr_B)^2}{ct^2}}
\|a\|^2_{L^{r}(w^{\frac{r}{2}},\,B)}\,\frac{dt}{t^3}\r]^{\frac{p}{2}}\\
&&\ls \sum_{j=3}^\fz \lf[w(2^jB)\r]^{1-\frac{p}{2}}\|a\|^p_{L^{r}(w^{\frac{r}{2}},\,B)}\noz\\
&&\hs\hs\times\lf[\int_0^{2^jr_B}
\lf(\frac{2^jr_B}{t}\r)^{2n\lf(\frac1r-\frac12+\frac3{2n}\r)}
e^{-\frac{(2^jr_B)^2}{ct^2}}\lf(\frac{1}{2^jr_B}\r)^{2n\lf(\frac1r-\frac12+\frac3{2n}\r)}
\,dt\r]^{\frac{p}{2}}\noz\\
&&\ls \sum_{j=3}^\fz \lf[2^{qnj}w(B)\r]^{1-\frac{p}{2}}\lf(2^j r_B\r)^{-p\lf(\frac{n}{r}-\frac{n}{2}+1\r)}
[w(B)]^{\frac{p}{2}-1}(r_B)^{p\lf(\frac nr-\frac n2+1\r)}\noz\\
&&\ls\sum_{j=3}^\fz 2^{qnj(1-\frac{p}{2})}2^{-jp\lf(\frac{n}{r}-\frac n2+1\r)}\ls 1.\noz
\end{eqnarray}
From an argument similar to that used in the above, it also follows that
\begin{eqnarray}\label{eq 3.x32}
{\rm I}_2\ls1.
\end{eqnarray}

Combining \eqref{eq 3.x32}, \eqref{eq 3.x31}, \eqref{eq 3.x3}, \eqref{eq 3.x2},
\eqref{eq 3.x1} and \eqref{eq 3.xx}, we obtain \eqref{eq 3.3}.
This further implies \eqref{eq 3.0}.
Therefore, we have
\begin{eqnarray}\label{eq 4.x2}
\lf(H_{L_w,\,{\rm Riesz}}^p(\rn)\cap L^2(w,\,\rn)\r)\st\lf(H_{L_w}^p(\rn)\cap L^2(w,\,\rn)\r).
\end{eqnarray}
From \eqref{eq 4.x2} and \eqref{eq 4.x1}, we deduce that
\begin{eqnarray*}
\lf(H_{L_w,\,{\rm Riesz}}^p(\rn)\cap L^2(w,\,\rn)\r)=\lf(H_{L_w}^p(\rn)\cap L^2(w,\,\rn)\r).
\end{eqnarray*}
This, together with \eqref{eq 4.x0} and \eqref{eq 3.0}, implies that
$H_{L_w,\,{\rm Riesz}}^p(\rn)\cap L^2(w,\,\rn)$ and $H_{L_w}^p(\rn)\cap L^2(w,\,\rn)$
coincide with equivalent quasi-norms.
Then, by a density argument, we complete the proof of
Theorem \ref{thm main}.
\end{proof}

\subsection*{Acknowledgment}
\hskip\parindent  The authors would like to thank Dr. Jun Cao,
Dr. Sibei Yang and Professor Renjin Jiang for some helpful conversations on this topic.

%%%%%%%%%%%%%%%%%%%%%%%%%%%%%%%%%%%%%%%%%%%%%%%%%%%%%%%%%%%%%%%%%%%%%%%%%%%%%%%%%%%%%%%%%%%%%%%%%%%%

\bigskip

\noindent{\small\sc Dachun  Yang  and Junqiang Zhang (Corresponding author)}

\medskip

\noindent{\small School of Mathematical Sciences, Beijing Normal University,
Laboratory of Mathematics and Complex Systems, Ministry of Education,
Beijing 100875, People's Republic of China}

\smallskip

\noindent{\it E-mails}: \texttt{dcyang@bnu.edu.cn} (D. Yang)

\hspace{1.12cm}\texttt{zhangjunqiang@mail.bnu.edu.cn} (J. Zhang)

\end{document}